\documentclass{amsart}[12pt]

\usepackage{amsmath, amsthm, amssymb, graphicx, tikz, zref}
\usetikzlibrary{arrows}
\usepackage{enumitem}
\usepackage{xcolor}
\newtheorem{theorem}{Theorem}[section]

\newtheorem{lemma}[theorem]{Lemma}
\newtheorem{example}[theorem]{Example}

\newtheorem{proposition}[theorem]{Proposition}
\newtheorem{conjecture}[theorem]{Conjecture}
\newtheorem{definition}[theorem]{Definition}
\newtheorem{remark}[theorem]{Remark}

\newcommand{\comments}[1]{}

\newcommand{\stab}[2]{\textnormal{Stab}_{#1} (#2)}
\newcommand{\cp}[1]{\text{Crit}_p (#1)}
\newcommand{\cv}[1]{\text{Crit}_v (#1)}
\newcommand{\vc}{\mathcal{L}}
\newcommand{\mc}{\mathcal{M}}
\newcommand{\vt}{\mathcal{T}}

\newcommand{\morph}[1]{\mathcal{M} (#1)}
\newcommand{\Hom}{\textnormal{Hom}}
\newcommand{\Ext}{\textnormal{Ext}}
\newcommand{\cone}[1]{\textnormal{cone} (#1)}
\newcommand{\cones}{\textnormal{cone}}

\newcommand{\support}[1]{\textnormal{supp} (#1)}
\newcommand{\gsupport}[2]{\textnormal{supp}^{#2} (#1)}
\newcommand{\fs}[2]{\mathcal{F}_{#2}(#1)}
\newcommand{\pfs}[2]{\mathcal{A}_{#2}(#1)}
\newcommand{\cores}{W}

\begin{document}

\title{On 2d-4d motivic wall-crossing formulas}

\author{Gabriel Kerr and Yan Soibelman}

\begin{abstract}
In this paper we propose definitions and examples of categorical enhancements 
of the data involved in the $2d-4d$ wall-crossing formulas which generalize both Cecotti-Vafa and Kontsevich-Soibelman motivic wall-crossing formulas. 
\end{abstract}
\maketitle
\section{Introduction}
In the foundational work \cite{ks08} the authors proposed a categorical 
framework 
for motivic Donaldson-Thomas theory.  Among other things they established a 
multiplicative wall-crossing formula which shows how the Donaldson-Thomas 
invariants change when the stability condition on the category crosses a real 
codimension one subvariety called {\it wall}.

It was also explained in the loc. cit. that the wall-crossing 
formulas themselves do not require the categorical framework. They arise when 
one has a much simpler structure consisting of a Lie algebra over 
$\mathfrak{g}=\oplus_{\gamma \in \Gamma}\mathfrak{g}_{\gamma}$ over ${\mathbb 
Q}$ graded by a free abelian group $\Gamma$ (lattice) and endowed with the 
so-called {\it stability data}. The latter consist of the homomorphism of 
abelian groups $Z:\Gamma \to {\mathbb C}$ (central charge) and a collection of 
elements $a(\gamma)\in \mathfrak{g}_{\gamma}, \gamma\in  \Gamma-\{0\}$. The 
stability data satisfy one axiom called Support Property. We are going to recall 
the details in the next section. Having stability data one can define for each 
strict sector $V\subset {\mathbb R}^2$ a pronilpotent Lie algebra 
$\mathfrak{g}_V$ as well as the pronilpotent Lie group 
$G_V=exp(\mathfrak{g}_V)$. The element $A_V=exp(\sum_{Z(\gamma)\in 
V}a(\gamma))$ is well-defined as an element of $G_V$. For each ray $l\subset V$ 
one has an  element $A_l\in G_l\subset G_V$. The wall-crossing formula is in 
fact a collection of the following formulas, one for each strict sector $V$: 

$$A_V=\prod_{l\subset V}A_l.$$
Here the product in the RHS is taken in the clockwise order. It can contain 
countably many factors.

It was explained in the loc.cit. that the above data appear in many different 
frameworks.
In particular the Donaldson-Thomas theory of a $3$-dimensional Calabi-Yau 
category proposed in the loc. cit.  gives an example of a graded Lie algebra
endowed with stability data.

In the most abstract setting  the wall-crossing formulas appears appear in the 
following way.
Suppose one has a triangulated $A_\infty$-category $\mathcal{C}$ which is 
linear over some ground field. The
lattice $\Gamma$ is the $K$-group (or its quotient) and the Lie algebra 
$\mathfrak{g}$ is the motivic Hall algebra $H(\mathcal{C})$ defined in 
\cite{ks08}.
Choosing a stability structure on $\mathcal{C}$ one can define for each strict 
sector $V$ as above a subcategory $\mathcal{C}_V$ generated by semistables with 
the central charge in $V$. This gives the group $G_V$ as the group of invertible
elements in the natural completion of the associative algebra 
$H(\mathcal{C}_V)$.
 
Then 
the wall-crossing formula can be written  as the product
$$A^{Hall}_V=\prod_{l\subset V}A_l^{Hall}.$$

Each factor in the RHS represents the input of semistable objects of 
$\mathcal{C}$ with the central charge belonging to a particular ray $l$.
One can change the central charge without changing the LHS. Then we obtain the 
formula which gives the equality of products of inputs of semistable objects for 
two different stability structures. Intuitively one can think that in the 
simplest case those stability structures are separated by
a real codimension one ``wall''. This explains the term ``wall-crossing 
formula''.

The authors of \cite{ks08} considered an important class of categories for 
which the wall-crossing formula can be further simplified. This is the case 
when $\mathcal{C}$ is a $3$-dimensional
 Calabi-Yau category ($3CY$ category for short). 
In this case  there is an 
algebra homomorphism from each $H (\mathcal{C}_V)$ 
to the so-called motivic quantum torus.  It induces a 
wall-crossing formula in a much simpler algebra, namely in the motivic quantum 
torus. Those are wall-crossing formulas 
for motivic Donaldson-Thomas invariants of the category $\mathcal{C}$. There is 
a ``quantization parameter'' in the story, which is
the Lefshetz motive $\mathbb{L}$, i.e. the motive of the affine line. One can 
consider a ``quasi-classical limit'' as ${\mathbb L}\to 1$. It is well-defined 
under some assumptions.   
The underlying graded Lie algebra is the one of the 
Hamiltonian vector fields 
on the Poisson torus (or its quantization) associated with the quotient of 
$K_0(\mathcal{C})$. It is naturally graded by this quotient, and it is endowed 
with the integer skew-symmetric form  induced by the Euler form. Graded 
components of this Lie algebra are $1$-dimensional. Hence elements 
$a(\gamma)$ are rational numbers. It was explained in \cite{ks08}  that the 
inverse M\"obius transform produces a collection of numbers $\Omega(\gamma)$ 
which are ``tend'' to be integers. The number $\Omega(\gamma)\in \mathbb{Z}$ is 
called the numerical Donaldson-Thomas invariant of class $\gamma$. (DT-invariant 
for short). It measures the ``number'' of semistable objects of $\mathcal{C}$ 
having a fixed $K$-theoretical class $\gamma$.

The above class of wall-crossing formulas covers a variety of examples 
which include Calabi-Yau $3$-folds and quivers with potential. They are known in 
gauge theory as $4d$ wall-crossing formulas for BPS invariants (or refined BPS 
invariants in the case of quantum tori).

Another class of examples is known as $2d$ wall-crossing formulas. 
They are associated with finite-dimensional simple Lie algebras graded by their 
root lattices. Contrary to the $4d$ wall-crossing formulas they do not have an 
obvious categorical or motivic origin.

Finally, there are so-called $2d-4d$ wall-crossing formulas (see \cite{gmn12}) 
which also follow from the formalism of \cite{ks08}, but 
similarly to the $2d$ case without categorical or motivic examples.

Despite of various unfinished attempts to find the categorical framework for 
the $2d$ and $2d-4d$ wall-crossing formulas, 
the problem is widely open. One of the purposes of this paper is to propose 
such a categorical  framework  and to illustrate it 
in algebraic and geometric examples.

 {\it Acknowledgments.}  We thank to E. Diaconescu and M. Kontsevich for useful 
discussions on the topic. 
In particular, our geometric example is a result of the re-thinking of the 
physically motivated computations made by Diaconescu.
 Y.S. thanks to IHES for excellent research conditions. 
 His work was partially supported by an NSF grant and by Simu-Munson Star 
Excellence award of KSU. G.K. was partially supported by a Simons collaboration grant.

\section{Groupoids and stability data on graded Lie algebras}

\subsection{Groupoids and associated graded Lie algebras}
This section will describe a generalization of the definition of stability data 
as introduced 
in \cite[Section~2]{ks08} to groupoids. In fact, this generalization is in some 
sense a restricted version of the usual notion of stability data. To accomplish 
this generalization,  we begin by recalling the structures considered in 
\cite{gmn12}.  There, the authors introduced the following algebraic data which 
are used in the   2d-4d wall-crossing formulas discussed in the loc.cit. :
\begin{enumerate}
	\item A lattice $\Gamma$ with anti-symmetric pairing $\left< , \right>$.
	\item A connected groupoid $\mathbb{V}$ with objects $\mathcal{V} \sqcup 
\{o\}$ 
whose automorphism groups are canonically isomorphic to $\Gamma$. 
	\item A central charge $Z : \mathbb{V} \to \mathbb{C}$ which is a 
homomorphism 
of groupoids (where $\mathbb{C}$ is regarded as the groupoid with one object).
	\item A function $\Omega : \Gamma \to \mathbb{Z}$.
	\item A function $\mu : \morph{\mathbb{V}}^\circ \to \mathbb{Z}$ 
	where $\morph{\mathbb{V}}^\circ$ are the morphisms of $\mathbb{V}$ with 
distinct source and target.
	\item A 2-cocycle $\sigma \in C^2 (\mathbb{V}, \{\pm 1\} )$.
\end{enumerate} 
The corresponding wall-crossing formulas (WCF for short)  
incorporate both 2d Cecotti-Vafa WCF and 
4d Kontsevich-Soibelman WCF (non-motivic version). \footnote{In fact the framework of 
\cite[Section~2]{ks08} covers all types of wall-crossing formulas, including 
2d, 4d, and 2d-4d. However the tradition in physics literature is to call  2d 
wall-crossing formulas ``Cecotti-Vafa WCF'', while 4d wall-crossing formulas are
called ``Kontsevich-Soibelman WCF''.}

For this paper, we will begin  with a small groupoid $\mathbb{V}$ and 
its $2$-cocycle $\sigma$, building a generalized stability data from this point.
Following \cite{gmn12}, we will write the objects of $\mathbb{V}$ using the notation like $i,j,k$ and the  morphism sets $\Hom_{\mathbb{V}} (i,j)$ will be often denoted by $\Gamma_{ij}$. Composition in $\mathbb{V}$ will be written additively.  If $\gamma_1$ and $\gamma_2$ are not composable, we define $\sigma (\gamma_1 , \gamma_2) = 0$.

In our setting, we will allow for disconnected groupoids (i.e. groupoids with 
potentially 
more than one isomorphism class). Write $\mathbb{W}$ for the unique groupoid 
equivalent to $\mathbb{V}$ for which, given distinct $i, j \in \mathbb{W}$, 
$\Hom_{\mathbb{W}} (i,j) = \emptyset$. Fix  equivalences
\begin{align}
F : \mathbb{V} & \to \mathbb{W},  & G : \mathbb{W} & \to \mathbb{V}, & \eta : GF & \Rightarrow 1_\mathbb{V}. 
\end{align}

Taking $\mathbb{C}$ to be the groupoid with one object,  morphisms complex 
numbers and composition equal to addition, 
we will frequently consider the additive group of groupoid homomorphisms 
$\Hom_{grpd} ( \mathbb{V}, \mathbb{C})$. We now give an alternative 
characterization of this group. For every $i \in Ob (\mathbb{W})$, consider the 
groups
\begin{align*}
\Gamma_i & = \Hom_{\mathbb{W}} (i,i), \\
W_i & = \left\{ \sum a_j e_j \in \mathbb{Z}^{F^{-1} (i)} : \sum a_j = 0 \right\},
\end{align*} 
and define
\begin{align*}
\Gamma_\mathbb{V}:=  \bigoplus_{i \in Ob (\mathbb{W})} \left( \Gamma_i \oplus W_i \right).
\end{align*}
We note that $\Gamma_\mathbb{V}$ does not depend on the equivalence $F$. We take 
\begin{align*} \pi_W : \Gamma_\mathbb{V} \to \oplus_{i \in Ob (\mathbb{W})} W_i \end{align*} to be the projection.
\begin{lemma}
	There is an isomorphism of abelian groups 
	\begin{align}
	\Hom_{grpd} (\mathbb{V}, \mathbb{C}) \cong \Hom_{Ab} ( \Gamma_\mathbb{V} , \mathbb{C} ).
	\end{align}
\end{lemma}
\begin{proof}
For $j \in Ob (\mathbb{V})$ let $w_j = e_j - e_{GF j} \in W_{Fj}$ and note that $\{w_j : j \in F^{-1}(i)\}$ forms a basis for $W_i$. As $G$ is an equivalence, we have that $G : \Gamma_i \stackrel{\cong}{\longrightarrow} \Gamma_{Gi Gi}$.
Define the map \[ f: \Hom_{grpd} (\mathbb{V}, \mathbb{C}) \to \Hom_{Ab} ( \Gamma_\mathbb{V} , \mathbb{C} ) \]
by 
\begin{align}
f(Z) (w_j) & = Z (\eta_j) & f(Z) (\gamma) & = Z (G (\gamma)).
\end{align}
Here  $\mathbb{C}$ means the groupoid with one object and space of morphisms equal to the field of complex numbers.
An elementary check shows this to be an isomorphism of groups.
\end{proof}
When $\mathbb{V}$ is finite may extend the elements $w_j$ into a type $A$ root system. Generally, we take
\begin{align*}R_i & = \{e_j - e_k \in W_i\} ,& R_\mathbb{V} & = \oplus_{i \in \mathbb{W}} R_i .\end{align*}
We can now define a groupoid graded Lie algebra.
\begin{definition}
	A $\mathbb{V}$-graded Lie algebra $\mathfrak{g}$ is a $\Gamma_\mathbb{V}$-graded Lie algebra
	\[ \mathfrak{g} = \oplus_{\gamma\in \Gamma_\mathbb{V}} 
\mathfrak{g}_\gamma \] satisfying $\mathfrak{g}_\gamma \ne 0$ implies $\pi_W 
(\gamma ) \in R_\mathbb{V}$.
\end{definition} 
In the case where  $Ob (\mathbb{V})$ is finite, this implies that the grading has a map to a product of type $A$ root systems.

\begin{example} \label{ex:poistorus} Write $\mathbb{K}[\mathbb{V}]$ for twisted 
groupoid 
algebra of $\mathbb{V}$ defined by $\sigma$.  In particular, 
$\mathbb{K}[\mathbb{V}]$ has a basis $\{e_\gamma : \gamma \in \Gamma_{ij}\}$ 
with multiplication 
\begin{align*}
e_{\gamma_1} e_{\gamma_2} = \sigma (\gamma_1, \gamma_2)  e_{\gamma_1 + \gamma_2}.
\end{align*}
whenever $\gamma_1$ and $\gamma_2$ are composable and zero otherwise.

Given an anti-symmetric integer pairing $\left< , \right>$ on $\oplus_{j \in \mathbb{W}} \Gamma_j$, orthogonal on distinct summands, we may define compatibility of the pairing and $\sigma$ as 
\[ \sigma (\gamma_1, \gamma_2) = (-1)^{\left<F(\gamma_1), F(\gamma_2)\right>} \] 
for all $i \in Ob (\mathbb{V})$ and $\gamma_1, \gamma_2 \in \Gamma_{ii}$.

To define a $\mathbb{V}$-graded Lie algebra, if $\gamma \in \Gamma_{ij}$, 
we take the degree of $e_\gamma$ to be
\begin{align}
| e_\gamma | := F(\gamma) + e_j - e_i \in \Gamma_{\mathbb{V}}.
\end{align} 
and define the bracket
\begin{align} 
[e_{\gamma_1}, e_{\gamma_2}] = \sigma (\gamma_1, \gamma_2) \left< F(\gamma_1), F (\gamma_2) \right> \left( e_{\gamma_1 + \gamma_2} + e_{\gamma_2 + \gamma_1} \right).
\end{align}
\end{example}

\subsection{${\mathbb V}$-stability data}

The generalization of stability data of a $\Gamma$-graded Lie algebra (see 
\cite{ks08}) to the $\mathbb{V}$-graded case is straightforward.

\begin{definition}
	Let $\mathbb{V}$ be a groupoid with finitely many objects.
	Given a $\mathbb{V}$-graded Lie algebra $\mathfrak{g}$, stability data on $\mathfrak{g}$ is a pair $(Z, a)$
	\begin{enumerate}
		\item  $Z : \mathbb{V} \to \mathbb{C}$ is a homomorphism of groupoids, 
		\item  $a = \left( a ( \gamma ) \right)_{\gamma \in \Gamma_\mathbb{V}}$ is a collection of elements in $\mathfrak{g}$ where $a(\gamma) \in \mathfrak{g}_\gamma$.
	\end{enumerate}
	This data must satisfy the following {\bf Support Property}
	\begin{itemize}
		\item[] For a norm $\| \|$ on $\Gamma_\mathbb{V} \otimes \mathbb{R}$, there exists $C > 0$ such that for any $\gamma \in \support{a}$ we have
		 \begin{align} \label{eq:bounds} \| \gamma\| \leq C Z (\gamma ). \end{align}
	\end{itemize}
\end{definition}
Equivalently, the last condition means the following:  there exists a quadratic form 
$Q$ on $\Gamma_\mathbb{V} \otimes \mathbb{R}$ such that it is non-positive on 
$\ker (Z)$ and non-negative on the set of $\gamma$'s such that $a(\gamma)\ne 0$ 
(sometimes this set is called the {\it support} of stability data).

In other words, groupoid stability data is precisely stability on the 
$\Gamma_\mathbb{V}$-graded 
Lie algebra $\mathfrak{g}$ in the sense of \cite[Section~2.1]{ks08}. Fixing 
the quadratic form $Q$  we may define 
the subset of stability data $\stab{Q}{\mathfrak{g}}$ as in 
\cite[Section~2.2]{ks08}. By applying Theorems $2$ and $3$ of loc. cit., we may 
assert the following facts concerning $\stab{}{\mathfrak{g}}$:
\begin{enumerate}
	\item The set of stability data $\stab{}{\mathfrak{g}}$ 
	equals $\cup_{Q} \stab{Q}{\mathfrak{g}}$.
	\item The set of stability data $\stab{}{\mathfrak{g}}$ 
	can be endowed with a Hausdorff topology.
	\item For a quadratic form $Q_0$ and $Z_0 : \Gamma_\mathbb{V} \to \mathbb{C}$, let \[ U_{Q_0, Z_0} = \{Z \in \Hom_{Ab} (\Gamma_\mathbb{V}, \mathbb{C}) : Q_0 \text{ is negative definite on }\ker Z_0 \}. \]
	Then there exists a local homeomorphism $U_{Q_0, Z_0} \to 
\stab{Q_0}{\mathfrak{g}}$ 
which is a section of the projection to $\Hom_{Ab} (\Gamma_\mathbb{V}, 
\mathbb{C})$.
\end{enumerate}
As a result, for any strict sector $V$ in $\mathbb{C}$ (i.e. the one which does 
not contain a straight line) one can define an element 
$A_V=exp(\sum_{Z(\gamma)\in V, Q(\gamma)\ge 0}a(\gamma))$  of the pro-nilpotent 
group $G_{V,Z, Q}$ corresponding to the pro-nilpotent Lie algebra 
$\mathfrak{g}_{V,Z,Q}$. This satisfies the {\bf factorization property}
\[ A_V = A_{V_1} A_{V_2}  \]
where $V = V_1 \sqcup V_2$ in clockwise order.


\section{Categorification of 2d-4d wall-crossing formulas}

\subsection{Motivic correspondences and motivic Hall algebras} \label{sec:motcor}
Given an ind-constructible, pre-triangulated, $A_\infty$-category $\mathcal{C}$ 
over $\mathbb{K}$, 
we recall from \cite[Section~6.1]{ks08} the definition of the motivic Hall 
algebra. We refer to the loc.cit. about a convenient introduction to the 
motivic functions on algebraic varieties and algebraic stacks.

We will assume all our categories $\mathcal{C}$ are 
ind-constructible triangulated $A_{\infty}$-categories with  finite 
dimensional $Hom's$. We first have that $Ob (\mathcal{C})$ is an 
ind-constructible stack which is a countable disjoint union of quotient stacks 
\[ Ob (\mathcal{C} ) = \sqcup_{i = 1}^\infty Y_i/GL (N_i) \] for some $N_i\ge 1$. 

The motivic Hall algebra $H(\mathcal{C})$ is then the $Mot (Spec 
(\mathbb{K}))[\mathbb{L}^{-1}]$-module $\oplus_i Mot_{st} 
(Y_i, GL (N_i))[\mathbb{L}^{-1}]$, where  $Mot_{st} 
(Y_i, GL (N_i))$ denote the abelian group of motivic functions on the quotient 
stack $Y_i/GL (N_i)$. We recall that, given $E, F \in 
\mathcal{C}$, the 
$N$-truncated Euler characteristic is given as 
\begin{align*}
(E, F)_{\leq N} := \sum_{i\leq N} (-1)^i \dim_{\mathbb{K}} (\Hom_\mathcal{C} (E, F[i])).
\end{align*}

The product structure in $H (\mathcal{C})$ is defined such as follows. 
For any two constructible families $\pi_1: X_1 \to Ob (\mathcal{C}) , \pi_2 
: X_2 \to Ob (\mathcal{C})$ and $n \in \mathbb{Z}$ 
the authors of \cite{ks08} define the spaces
\begin{align*}
W_{n} = \{(x_1, x_2, \alpha) | x_i \in X_i,  \alpha \in \Ext^1 (\pi_2 (x_2), \pi_1 (x_1)), (\pi_2 (x_2), \pi_1 (x_1) )_{\leq 0} = n\}.
\end{align*}
The element $[W_n \to Ob (\mathcal{C})] \in H(\mathcal{C})$ is defined 
as taking  the cone of $\alpha : \pi_2 (x_2)[-1] \to \pi_1 (x_1)$. Then the 
product on $H (\mathcal{C})$ is given as 
\begin{align*}
[\pi_1: X_1 \to Ob (\mathcal{C}) ] \cdot [\pi_2 : X_2 \to Ob (\mathcal{C})] : = \sum_{n \in \mathbb{Z}} [W_n \to Ob (\mathcal{C})] \mathbb{L}^{-n}
\end{align*}
It is shown in loc. cit. that it makes $H(\mathcal{C})$ into an associative 
algebra.

We now introduce a more general version of motivic Hall algebras. 
Let $\iota: \cores \to Ob(\mathcal{C}) \times Ob ( \mathcal{C} )$ be a homogeneous 
ind-constructible affine subbundle of $Ext^*$ (with possibly empty fibers).  We 
will call $\cores$ a \textit{motivic correspondence} of $\mathcal{C}$.  Let 
$\support{\cores}$ be the substack of $Ob (\mathcal{C}) \times Ob (\mathcal{C})$ over 
which the fibers of $\cores$ are non-empty. Given $(A, B) \in \support{\cores}$ define the 
{\bf degree} $\deg_{A,B} (\cores )$ of $\cores$ over $(A,B)$ to be $k$ if $\iota^* (A, B) \subset \Ext^k (A,B)$. There may be some ambiguity here if $\iota^* (A,B) = 
{0}$, but we will take $0$ to be an element of any degree.  We will take $\cones 
: \cores \to Ob (\mathcal{C})$ to be the restriction of the map taking $\alpha \in 
\Ext^k (A, B)$ to the cone of $\alpha : A[-k] \to B$. Let $\nu : \support{\cores} \to 
\mathbb{Z}$ be a constructible function, which we call a \textit{weight}, so 
that $\support{\cores} = \sqcup_{n \in \mathbb{Z}} \gsupport{\cores}{n}$ and $\cores = 
\sqcup_{n \in \mathbb{Z}} \cores_n$ where $\cores_n$ is the fiber of $\iota$ over 
$\gsupport{\cores}{n}:=\nu^{-1}(n)$. Then, using the decomposition from $\nu$, one 
may decompose as before \[[(\pi_1 \times \pi_2)^* (\cores) \to Ob (\mathcal{C}) 
\times Ob (\mathcal{C})] = \sum_{n \in \mathbb{Z}} [(\pi_1 \times \pi_2)^* (\cores_n) 
\to Ob (\mathcal{C}) \times Ob (\mathcal{C})]. \] 
Taking cones of both sides gives the decomposition 
\[ [(\pi_1 \times \pi_2)^* (\cores) \to Ob (\mathcal{C})] = \sum_{n \in \mathbb{Z}} [(\pi_1 \times \pi_2)^* (\cores_n) \to  Ob (\mathcal{C})].  \]
The pair $(\cores, \nu)$ induces a product in the motivic Hall algebra as in the case 
where $\cores = \mathcal{E}xt^1$. 
In particular, one takes
\begin{align*}
[\pi_1: X_1 \to Ob (\mathcal{C}) ] \cdot_{(\cores, \nu)} [\pi_2 : X_2 \to Ob (\mathcal{C})] : = \sum_{n \in \mathbb{Z}} [(\pi_1 \times \pi_2)^* (\cores_n) \to Ob (\mathcal{C})] \mathbb{L}^{-n}
\end{align*}

\begin{definition}
	We call the pair $(\cores, \nu)$ an associative correspondence if $\cdot_{(\cores, \nu)}$ is an associative product in $H (\mathcal{C})$.
\end{definition}

Given an associative correspondence $(\cores, \nu)$, we may consider the algebra $H_{(\cores,\nu)} (\mathcal{C})$ which is identical to $H (\mathcal{C})$ as a vector space but endowed with product $\cdot_{(\cores, \nu)}$. 
\begin{figure}
	\begin{center}
		\begin{picture}(0,0)%
		\includegraphics{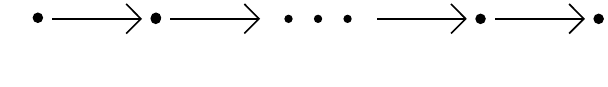}%
		\end{picture}%
		\setlength{\unitlength}{4144sp}%
		\begin{picture}(2766,391)(3361,-1160)
		\put(3486,-1096){\makebox(0,0)[lb]{\smash{$n$}}}
		\put(5531,-1096){\makebox(0,0)[lb]{\smash{$2$}}}
		\put(6041,-1096){\makebox(0,0)[lb]{\smash{$1$}}}
		\put(3926,-1096){\makebox(0,0)[lb]{\smash{$n - 1$}}}
		\end{picture}%
		\caption{\label{fig:An} The $A_n$ quiver}
	\end{center}
\end{figure}

\begin{example} \label{ex:An1}
Let $Q$ be the $A_n$ quiver with vertices $Q_0 = \{1, \cdots, n\}$ and arrows $Q_1 = \{ i \to (i -1) : 1 < i \leq n\}$ oriented as in Figure~\ref{fig:An}. Take $\{u_i: 1 \leq i \leq n \}$ to be the standard basis for the dimension vectors $\mathbb{N}^{Q_0}$. Let $\mathcal{A}$ be the dg category of representations of the path algebra $\mathbb{K} Q$. Let $\mathcal{B}$ be the full subcategory of indecomposable representations (concentrated in degree zero). By Gabriel's Theorem, the objects of $\mathcal{B}$ have dimension vectors 
\begin{align}
R^+ = \left\{ u_{ij}:= \sum_{k = i+ 1}^{j} u_k : 0 \leq i < j \leq n \right\}.
\end{align}
These correspond to positive roots of the $A_n$-root system. Write $M_{ij}$ for the corresponding indecomposable module and consider the correspondence
\begin{align}
\cores = \{(M_{kl}, M_{ij}, \gamma) : \gamma \in \Ext^1 (M_{ij}, M_{kl}), \cone{\gamma} \text{ is indecomposable} \}.
\end{align} 
Take $\nu : Ob (\mathcal{A}) \times Ob (\mathcal{A}) \to \mathbb{Z}$ to be constant and equal to zero. Then we claim that $(\cores, \nu)$ is an associative correspondence. 

To verify this, first one makes a basic computation to see
\begin{align} \label{eq:hom}
\Hom^* (M_{ij}, M_{kl}) \cong  \begin{cases}
\mathbb{K}\cdot \alpha  &  i \leq k < j \leq l, \\
(\mathbb{K} \cdot \beta)[-1] &  k < i \leq l < j, \\
0 & \text{otherwise.}
\end{cases} 
\end{align}
Furthermore, with the indices as in equation~\eqref{eq:hom} and setting $M_{ii} = 0$ for any $i$, we have
\begin{align}
\cone{\alpha} & \cong M_{jl} \oplus M_{ik}[1] \\
\cone{\beta} & \cong M_{il} \oplus M_{kj}.
\end{align}
Thus we may rewrite $\cores$ as 
\begin{align}
\cores = \{(M_{ki},M_{ij}, \beta): 0 \leq k < i < j \leq n\}.
\end{align}
where cone maps $(M_{ki}, M_{ij}, \beta)$ to $M_{kj}$.

For a $\mathbb{K}$-algebra $R$, let $U_n (R)$ be the algebra of strictly upper triangular matrices with entries in $R$. Then there is a natural algebra homomorphism 
\begin{align}
\int: H_{(\cores, \nu)} (\mathcal{A}) \to U_n (Mot (Spec (\mathbb{K})))
\end{align}
taking $[ \pi : X \to M_{ij}]$ to the upper triangular matrix $[X]$ in the $i$-th row and $j$-th column.
\end{example} 

\subsection{Motivic enhancements of ind-constructible categories}

We will need a relative version of an associative correspondence in the next 
section. 
For this, let $\mathcal{C}$ be an ind-constructible triangulated 
$A_\infty$-category, $\mathcal{D}$ be a small strict category and 
$\mathcal{M} (\mathcal{D})$ the set of morphisms in $\mathcal{D}$. We assume 
that $\mathcal{M} (\mathcal{D})$ is an ind-constructible stack.

\begin{definition}
	 Let $\Phi = (\{\mathcal{C}_{ij}, \epsilon_{ij}\}_{(i,j) \in Ob(D)^2}, \{(\cores_{ijk} , \nu_{ijk})\}_{(i,j,k) \in Ob (\mathcal{D})^3}) $ consist of the following structures:
	 \begin{enumerate}
	 	\item  A collection of full ind-constructible subcategories $\{\mathcal{C}_{ij}\}_{(i,j) \in Ob(D)^2}$ of $\mathcal{C}$.
	 	\item For every $(i,j) \in Ob (D)^2$, an ind-constructible morphism $\epsilon_{ij} : Ob ( \mathcal{C}_{ij} ) \to \Hom_{\mathcal{D}} (i,j)$.
	 	\item A weighted motivic correspondence $(\cores_{ijk} , \nu_{ijk})$ of $\mathcal{C}$ for every triple of objects in $\mathcal{D}$.
	 \end{enumerate} 
	 We say that $\Phi$ is a  motivic  enhancement of $\mathcal{D}$ by $\mathcal{C}$ if:
	 \begin{enumerate} 
	 	\item[(i)] for every triple $i,j,k \in Ob (\mathcal{D})$, $\mathcal{C}_{jk} \times \mathcal{C}_{ij} \subset \support{\cores_{ijk}}$,
	 	\item[(ii)] the product $[\pi_1 : X_1 \to \mathcal{C}_{jk}] \cdot [\pi_2 : X_2 \to \mathcal{C}_{ij}]$ induced by $(\cores_{ijk}, \nu_{ijk})$ yields a motive over $\mathcal{C}_{ik}$,
	 	\item[(iii)] the product induced by the correspondences $(\cores_{ijk}, \nu_{ijk})$ is associative. 
	 \end{enumerate}
\end{definition}
\begin{example} \label{ex:An2.5}
Take $\mathcal{D}$ to be the directed category with objects $\{0, \ldots, n\}$ and a morphism $i \to j$ if and only if $i \geq j$. Let $\mathcal{A}$ be the dg category of representations of the $A_n$-quiver as in Example~\ref{ex:An1} and take $\mathcal{C}_{ij} = \{M_{ji}[r] : r \in \mathbb{Z}\}$. For correspondences, $\cores_{ijk}$ take $\{(\beta[s - r] , M_{ij}[r], M_{jk}[s]) :r, s \in \mathbb{Z}\}$. In general, one may take non-constant $\nu_{ijk}$, but for simplicity consider $\nu_{ijk} = 0$. This enhancement ``spreads out'' the associative correspondence from the previous example across $\mathcal{D}$.
\end{example}
The fundamental invariant associated to a motivic  enhancement will now be defined.
\begin{definition}
	Let $\Phi =(\{\mathcal{C}_{ij}, \epsilon_{ij}\}, \{(\cores_{ijk} , \nu_{ijk})\}) $ be a motivic  enhancement 
	 of $\mathcal{D}$ by $\mathcal{C}$. The motivic Hall category 
\begin{align} H_\Phi (\mathcal{D})
	 \end{align} of the enhancement $\Phi$ is 
	 the category with objects $Ob (\mathcal{D})$ and morphisms 
	\begin{align} \Hom_{H_\Phi (\mathcal{D})} (a, b) = H (\mathcal{C}_{(a,b)}). \end{align}
\end{definition}

We will consider the following main example of motivic  enhancements throughout the paper. Let $\Lambda^\prime$ and $\Lambda$ be ind-constructible abelian groups and 
\begin{align} \phi : \Lambda^\prime \to \Lambda \end{align}
a group homomorphism. Suppose that $A$ is a subset of $\Lambda$. 
\begin{definition} \label{defn:gpdphi}
	The \textit{groupoid induced by $\phi$ and $A$}, which we will denote by 
$\Lambda_{\phi, A}$ (or simply $\Lambda_{\phi}$), has objects equal $A$ and, given $\lambda_1, 
\lambda_2$ in $A$
	\begin{align} \label{eq:deflamphi} \Hom_{\Lambda_\phi} (\lambda_1 , \lambda_2 ) = \phi^{-1} (\lambda_2 - \lambda_1). \end{align}
	Composition is addition in $\Lambda^\prime$.
\end{definition}
Note that if, for $\lambda_1, \lambda_2 \in A$,
\begin{align} \label{eq:Acondition} \lambda_2 - \lambda_1 \not\in \text{im} (\phi) \end{align}
then $\Hom_{\Lambda_\phi} (\lambda_1 , \lambda_2 ) = \emptyset$.

To obtain a motivic enhancement, suppose $\chi : K_0 (\mathcal{C}) \to \Lambda^\prime$ is an additional homomorphism.  For any $\lambda \in \Lambda$ let 
\begin{align} \mathcal{C}_\lambda = \{X \in \mathcal{C} : (\phi \circ \chi ) ([X]) = \lambda \}. \end{align}
Let $\cores$ be the motivic correspondence $\Ext^1$ and $\nu$ the $N$-truncated Euler characteristic with $N = 0$. 
\begin{definition} \label{defn:phimotext} The triple $\Phi = (\{\mathcal{C}_{\lambda_2 - \lambda_1}\}_{(\lambda_1, \lambda_2) \in A^2}, \cores, \nu)$ is called the motivic  enhancement induced by $\phi, \chi$ and $A$.
\end{definition} 
In cases when $\Lambda^\prime = K_0 (\mathcal{C})$ and $\chi$ is the identity, we will simply say that $\Phi$ is induced by $\phi$ and $A$.

In this instance the motivic correspondence is independent of the triple of objects. An extreme version of this is taking $\Lambda = \{0\}$ and obtaining a motivic Hall category $H_\Phi ( \Lambda)$ with one object whose endomorphism ring is the motivic Hall algebra $H( \mathcal{C})$. 

\begin{example} \label{ex:An2}
Consider the category $\mathcal{A}$ of representations of the $A_n$ quiver as in Example~\ref{ex:An1}. Then $K_0 (\mathcal{A}) = \mathbb{Z}^{Q_0}$ is isomorphic to the root lattice of type $A_n$ with roots $R = R^+ \cup -R^+$. There is a homomorphism 
\begin{align} \label{eq:phiAn}
\phi : K_0 (\mathcal{A}) \to \mathbb{Z}^{n + 1} 
\end{align}
taking the root $u_{ij}$ to $e_j - e_i$. No two elements in the standard basis $A = \{e_0, \ldots, e_{n}\}$ satisfies equation \eqref{eq:Acondition}. Thus the category $\mathbb{Z}^{n + 1}_{\phi, A}$ is the groupoid that has objects $\{0, \ldots, n\}$ (with the correspondence $i \mapsto e_i$) and exactly one morphism $u_{ij}: i \to j$. We will refer to such groupoids as \textit{trivial groupoids} in that they are equivalent to a trivial group. Letting $\mathcal{A}_0$ be the full subcategory of $\mathcal{A}$ whose objects have trivial $K_0 (\mathcal{A})$ classes, we note that 
\begin{align} \label{eq:mot2cat} \mathcal{C}_{ij} = \begin{cases} M_{ij} + \mathcal{A}_0 & \text{if } i \leq j, \\
M_{ij}[1] + \mathcal{A}_0 & \text{otherwise}. \end{cases} 
\end{align}
\end{example}

To show the flexibility of this construction, we define a motivic  enhancement which has correspondences that are highly dependent on the objects. Let $Q = (Q_0, Q_1)$ be any quiver and $\mathcal{D}_Q$ the path category of $Q$. By this we mean that the objects of $\mathcal{D}_Q$ are the vertices $Q_0$ of $Q$ and, given $i, j \in Q_0$, the morphisms $\Hom_{\mathcal{D}_Q} (i,j)$ is the set of directed paths from $i$ to $j$ in $Q$. Let $\{(\mathcal{C}_{ij}, \epsilon_{ij})\}_{(i,j) \in Q_0^2}$ be any collection of constructible subcategories of $\mathcal{C}$, partitioned by arrows in $Q_1$ from $i$ to $j$. Given such an $a \in Q_1$ let $\mathcal{C}_a$ be the subcategory $\epsilon_{ij}^{-1} (a)$.  Take $\mathcal{\cores} = \{(\cores_{ijk}, \nu_{ijk})\}$ to be a collection of motivic correspondences and, given a sequence of composable arrows
\[ i_1 \stackrel{a_1}{\longrightarrow} i_2 \stackrel{a_2}{\longrightarrow} \cdots  \stackrel{a_n}{\longrightarrow} i_{n + 1} \]
in $\mathcal{D}_Q$, let $E_r \in \mathcal{C}_{a_r}$ for $r \in \{1,\ldots , n\}$. We will assume that for any sequence of morphisms $\gamma_i \in \cores_{i(i+1)(i + 2)} \cap \Ext^* (E_i, E_{i + 1})$, we have that \begin{align} \label{eq:vanishing} \mu_{\mathcal{C}}^n (\gamma_n , \ldots, \gamma_1 ) = 0 . \end{align} 
Here the map $\mu_\mathcal{C}^n$ is the $n$-th multiplication map in the 
$A_\infty$-structure of $\mathcal{C}$.

\begin{definition}
We call  motivic enhancement minimal if the correspondence over every pair 
of composable morphisms has minimal dimension.

\end{definition}

\begin{proposition}
	Let $Q$ be a simply-laced quiver and  $\Phi^\prime = (\{\mathcal{C}_{a}\}_{a \in Q_1},\{(\cores_{ijk}, \nu_{ijk})\} )$ be an assignment satisfying equation \eqref{eq:vanishing}. Then $\Phi^\prime$  has a minimal motivic  enhancement.
\end{proposition}

\begin{proof}
An enhancement must be made so that there is a assignment of categories to all paths of arrows (or morphisms in $\mathcal{D}_Q$).  Let $a_j : i_j \to i_{j + 1}$ be a collection of morphisms for $1 \leq j \leq n$ and $a =  a_n \cdots a_1$. Take
\begin{align*}
\mathcal{C}_a := \left\{ \left( \oplus_{j = 1}^n E_j [|a_j| - 1] , \sum_{j = 1}^{n - 1} \alpha_j \right) : E_j \in \mathcal{C}_{i_j}, \alpha_j \in \cores_{a_{j + 1}, a_j} \cap \Ext^* (E_j, E_{j + 1} ) \right\}.
\end{align*}
Here the notation $(E, \alpha)$ is for a twisted complex.
If $b_k : i_{n + k} \to i_{n + k + 1}$ for $1 \leq k \leq m$ is another collection of morphisms and $b = b_m \cdots b_1$, we take 
\begin{align*}
(\cores_{i_1i_{n +1}i_{n + k + 1}}, \nu_{i_1i_{n +1}i_{n + k + 1}}) := (\iota_{n} \circ \cores_{i_{n - 1}i_ni_{n + 1}} \circ \pi_1, \iota_{n} \circ \nu_{i_ni_{n + 1}i_{n + 2}} \circ \pi_1).
\end{align*}
More explicitly, for all $\alpha \in \cores_{i_ni_{n + 1}i_{n + 2}} \cap \Ext^* (E_n, F_1 )$, take the composition
\begin{align*}
\oplus_{j = 1}^n E_j \stackrel{\pi_n}{\longrightarrow} E_{n} \stackrel{\alpha}{\longrightarrow} F_1 \stackrel{\iota}{\longrightarrow} \oplus_{k = 1}^m F_k . 
\end{align*}
to yield a correspondence in $\cores_{i_1i_{n +1}i_{n + k + 1}}$. To obtain a precise characterization of the weight, write $\kappa : \cores_{i_ni_{n + 1}i_{n + 2}} \to \cores_{i_1i_{n +1}i_{n + k + 1}}$ for this composition. Over $(E_n, F_1)$, this is a linear map $\kappa_{(E_n, F_1)} : \Ext^* (E_n, F_1) \to \Ext^* (E,F)$. Furthermore, owing to the description of the categories $\mathcal{C}_a$, there are constructible functors from $G_a :\mathcal{C}_a \to \mathcal{C}_{a_n}$ and $G_b :\mathcal{C}_b \to \mathcal{C}_{b_1}$. The weight $\nu_{i_1i_{n +1}i_{n + k + 1}}$ is the pushforward of the weight $\nu_{i_ni_{n + 1}i_{n + 2}}$ in the sense that 
\begin{align*}
\nu_{i_1i_{n +1}i_{n + k + 1}} (E,F) := \nu _{(G_b (b), G_a (a))} (G_a(E) , G_b(F)) + \dim (\ker (\kappa_{(G_a (E), G_b (F))}).
\end{align*}
It immediately follows that if $[\rho_{E_j} : pt \to \mathcal{C}_{a_i}]$ are objects taking $pt$ to $E_j$, then 
\begin{align}
\rho_1 \cdots \rho_n = \int_{\alpha_j \in \cores_{j(j + 1) (j + 2)}} \mathbb{L}^{\sum_{i = 1}^{n - 1} \nu_{j(j + 1) (j + 2)} (\rho_j, \rho_{j + 1})} \rho_{(\oplus E_j, \sum \alpha_j)}.
\end{align}
\end{proof}
We note that one may prove a version of this theorem purely in the triangulated setting where associativity of compositions is directly tied to the octahedral axiom.  We conclude this section with an example illustrating this connection.

\begin{example} \label{ex:octahedral} Let $f_0: A_0 \to A_1$ and $f_1 : A_1 \to A_2$ be degree $0$ morphisms in $\mathcal{C}$ such that $f_1 f_0 = 0$. Taking $\delta: A_1 \to B = \cone{f_0}$ and applying the long exact sequence gives a morphism $h : B \to A_2$ such that $f_1 = h \delta$. Thus, $f_0$ and $f_1$ will generate a bottom ``cap'' 
 	\begin{equation} \label{eq:bottomcap}
 	\begin{tikzpicture}[baseline=(current  bounding  box.center), scale=1.5]
 	\node (A1) at (0,2) {$A_2$};
 	\node (A2) at (4,2) {$A_1$};
 	\node (C1) at (1, 1) {$\star$};
 	\node (C2) at (2,1) {$B$};
 	\node (C3) at (3,1) {$\star$};
 	\node (D1) at (0,0) {$A_3$};
 	\node (D2) at (4,0) {$A_0$};
 	\path[->,font=\scriptsize]
 	(D2) edge node[right] {$f_0$} (A2)
 	(A2) edge node[above] {$f_1$} (A1)
 	(A1) edge node[left] {$[1]$} node[right] {$f_2$} (D1)
 	(D1) edge node[below] {$[1]$} node[above] {$f_3$} (D2)
 	(D1) edge node[above] {} (C2)
 	(C2) edge node[above] {} (A1)
 	(C2) edge node[below] {$[1]$} (D2)
 	(A2) edge node[above] {$\delta$} (C2)
 	;
 	\end{tikzpicture} 
 	\end{equation}	
The octahedral axiom ensures the existence of the ``top'' cap
 	\begin{equation} \label{eq:topcap}
 	\begin{tikzpicture}[baseline=(current  bounding  box.center), scale=1.5]
 	\node (A1) at (0,2) {$A_2$};
 	\node (A2) at (4,2) {$A_1$};
 	\node (C1) at (2, 1.5) {$\star$};
 	\node (C2) at (2,1) {$C$};
 	\node (C3) at (2, 0.5) {$\star$};
 	\node (D1) at (0,0) {$A_3$};
 	\node (D2) at (4,0) {$A_0$};
 	\path[->,font=\scriptsize]
 	 	(D2) edge node[right] {$f_0$} (A2)
 	(A2) edge node[above] {$f_1$} (A1)
 	(A1) edge node[left] {$[1]$} node[right] {$f_2$} (D1)
 	(D1) edge node[below] {$[1]$} node[above] {$f_3$} (D2)
 	(C2) edge node[above] {} (D1)
 	(A1) edge node[above] {$[1]$} (C2)
 	(D2) edge node[below] {} (C2)
 	(C2) edge node[above] {$$} (A2)
 	;
 	\end{tikzpicture} 
 	\end{equation}	
 	
 Let $\mathcal{D}$ be the category with objects labeled by integers and a single morphism $f_{rs} : r \to s$ if and only if $r \leq s$. In particular, this is the path category associated to $Q = (\mathbb{Z}, Q_1)$ where $Q_1$ consists of arrows from $a_r : r \to r + 1$. For $r \in \mathbb{Z}$, write $\bar{r}$ for the congruence class of $r$ modulo $4$.
Define the subcategories
	\begin{align}
	\mathcal{C}_{a_r} = \{A_{\bar{r}} \}
	\end{align}
For the correspondences we simply take 
\begin{align}
\cores_{(a_{r - 1}, a_r)} = \{f_r \}
\end{align}
Then the first few terms of the  motivic enhancement can be sketched as below:
 	\begin{equation} \label{eq:diagram}
 	\begin{tikzpicture}[baseline=(current  bounding  box.center), scale=1.5]
 	\node (Z) at (-.5, 0) {$\cdots$};
 	\node (A0) at (0,0) {$\cdot_0$};
 	\node (A1) at (1.5,0) {$\cdot_1$};
 	\node (A2) at (3, 0) {$\cdot_2$};
 	\node (A3) at (4.5,0) {$\cdot_3$};
 	\node (A4) at (6,0) {$\cdot_4$};
 	\node (A5) at (6.5,0) {$\cdots$};
 	\path[->,font=\scriptsize]
 	(A0) edge node[above] {$A_0$} (A1)
 	(A1) edge node[above] {$A_1$} (A2)
 	(A2) edge node[above] {$A_2$} (A3)
 	(A3) edge node[above] {$A_3$} (A4)
 	(A0) edge[bend left=30] node[above] {$B$} (A2)
 	(A1) edge[bend left=30] node[above] {$C[1]$} (A3)
 	(A2) edge[bend left=30] node[above] {$B$} (A4)
 	(A0) edge[bend right=35] node[above=6pt, left=.5pt] {$A_3[1]$} (A3)
 	(A1) edge[bend right=35] node[above=6pt, right=.5pt] {$A_0[1]$} (A4)
 	;
 	\end{tikzpicture} 
 	\end{equation}	
\end{example}

\subsection{$\mathbb{V}$-stability conditions} \label{sec:vcat}

Let us start by defining 
a certain collection of subcategories of an ind-constructible,  triangulated 
$A_\infty$-category $\mathcal{C}$ (for example, the category of twisted 
complexes of an ind-constructible $A_\infty$-category). If $\mathcal{C}^\prime$ 
is a subcategory of $\mathcal{C}$, we write 
 \begin{align}
[\mathcal{C}^\prime] = \{[A] \in K_0 (\mathcal{C}): A \in Ob (\mathcal{C}^\prime)\}
 \end{align} for the collection of $K_0 (\mathcal{C})$ classes of objects in $\mathcal{C}^\prime$. Recalling the definition of $\deg_{A,B} (\cores_{ijk})$ from Section~\ref{sec:motcor}, we say that $\Phi = (\{C_{ij}\}, \{\cores_{ijk}\})$ is an \textit{odd motivic  enhancement} of $\mathcal{D}$,  if $\deg_{A,B} (\cores_{ijk})$ is odd for all $(A, B) \in \support{\cores_{ijk}}$.  Given an odd enhancement $\Phi$, we may define the Grothendieck category $K_\Phi (\mathcal{D})$ of the enhancement as the category with objects equal to $Ob (\mathcal{D})$ and morphisms 
 \[ \Hom_{K_\Phi (\mathcal{D})} (i, j) := [\mathcal{C}_{ij}]. \]
To justify that $K_\Phi (\mathcal{D})$ is well defined, assume $[A] : i \to j$ and $[B] : j \to k$ and let $r = \deg_{A, B} \cores_{ijk}$. Then, since the support of $\cores_{ijk}$ contains $\mathcal{C}_{jk} \times \mathcal{C}_{ij}$, it follows that there is some $\delta \in \cores_{ijk} \cap \Ext^r (A, B)$ such that $\cone{\delta} \in \mathcal{C}_{ik}$. Of course, $[\cone{\delta}] = [B] - [A[r]] = [B] + [A]$ so that composition of morphisms is well defined.

For the following definition, we do not require that $\mathbb{V}$ is connected but will assume that the morphisms of $\mathbb{V}$ form an ind-constructible stack. 

\begin{definition}\label{defn:Vcoll}
	Let $\mathcal{C}$ be an ind-constructible, triangulated 
$A_\infty$-category. A  motivic  enhancement of $\mathbb{V}$ by $\Phi = 
(\{\mathcal{C}_{ij}, \epsilon_{ij}\},\{(\cores_{ijk} , \nu_{ijk})\} )$ will be called 
a \textit{$\mathbb{V}$-collection} if 
	\begin{enumerate}
		\item $\mathcal{C}_{ij} = \mathcal{C}_{ji} [1]$, 
		\item for any $E \in \mathcal{C}_{ij}$, the morphism $Id[1] : E \to E[1]$ is in the support of $\cores_{iji}$.
		\item the morphisms $\epsilon_{ij}$ induce an isomorphism $\epsilon : K_\Phi (\mathbb{V}) \to \mathbb{V}$.
	\end{enumerate}
	We will say it is a \textit{standard} $\mathbb{V}$-collection if $\cores_{ijk}$ is the pullback of $\Ext^1$ along the inclusion $Ob( \mathcal{C}_{jk}) \times Ob ( \mathcal{C}_{ij})$ to $Ob (\mathcal{C}) \times Ob (\mathcal{C})$ and $\nu_{ijk}$ is the $N$-truncated Euler characteristic with $N = 0$.
\end{definition}
\begin{remark}
Property (1) implies that the categories $\mathcal{C}_{ij}$ are stable under even shifts. Property (2) asserts that one may think of $E[1] \in \mathcal{C}_{ji}$ as the inverse of $E \in \mathcal{C}_{ij}$. 
\end{remark}
\begin{remark}
While one frequently uses the notion of equivalence between categories, and rarely the notion of isomorphism, condition (3) in Definition~\ref{defn:Vcoll} relies on the strict notion of isomorphism. Indeed, the case of a trivial groupoid, i.e. one categorically equivalent to a trivial group, is still of substantial interest as it relates directly to the purely 2d WCF.
\end{remark}
We will often shorten the notation of a $\mathbb{V}$-collection of subcategories to $\{\mathcal{C}_{ij}: i,j \in \mathbb{V} \}$. A $\mathbb{V}$-collection of subcategories $\{\mathcal{C}_{ij} : i, j \in \mathbb{V}\}$ will be the central notion behind the categorification of 2d-4d wall-crossing data. To make this precise, we must first give definitions of a central charge, Harder-Narasimhan sequence and slicings in this context. 
\begin{definition} \label{defn:charge}
A $\mathbb{V}$-\textit{central charge} is a homomorphism of groupoids $Z : \mathbb{V} \to \mathbb{C}$ where $\mathbb{C}$ is taken as the groupoid with one object whose morphisms are $\mathbb{C}$.
\end{definition}
As seen in the previous section, one may easily describe central charges in the context of $\mathbb{V}$-collections. For example, consider a homomorphism $\phi : K_0 (\mathcal{C}) \to \Lambda$, $A \subset \Lambda$ and the groupoid $\Lambda_\phi$ defined in equation \eqref{eq:deflamphi}. The motivic  enhancement $\Phi$ induced by $\phi$ is a standard $\Lambda_\phi$-collection. A generic homomorphism 
\begin{align} \label{eq:inducedcharge}
\tilde{Z} : \Lambda \to \mathbb{C}
\end{align}
will yield a $\mathbb{V}$-central charge on $\Lambda_\phi$ by composing with $\phi$.

\begin{definition} \label{defn:preslice}
	Let $\Phi := \{\mathcal{C}_{ij} : i, j \in \mathbb{V} \}$ be a $\mathbb{V}$-collection of ind-constructible subcategories of $\mathcal{C}$ and $Z : \mathbb{V} \to \mathbb{C}$ a $\mathbb{V}$-central charge. A \textit{pre-slicing} $\mathcal{P}$ of $\Phi$ is a collection of ind-constructible subcategories $\mathcal{P}_{ij} (\theta ) \subset \mathcal{C}_{ij}$ for every $i, j \in \mathbb{V}$ and $\theta \in \mathbb{R}$. These must satisfy the following conditions:

	\begin{enumerate}[label=(\roman*), ref=\thetheorem(\roman*)]
		\item \label{defn:preslice:1} For all $\theta \in \mathbb{R}$ and $i, j \in \mathbb{V}$, $\mathcal{P}_{ij} (\theta + 1) = \mathcal{P}_{ji}(\theta)[1]$,
		\item \label{defn:preslice:2} For $i, j, k, l \in \mathbb{V}$, $\theta_1 > \theta_2$,  $E_1 \in \mathcal{P}_{ij} (\theta_1)$ and $E_2 \in \mathcal{P}_{kl} (\theta_2)$ 
		\[ \Ext^0 (E_1, E_2) = 0, \]
	\end{enumerate}
\end{definition}
We will often refer to an object $E \in \mathcal{P}_{ij} (\theta)$ as \textit{semi-stable}.  

From this point on, we assume that $\mathbb{V}$ is a finite groupoid. For a standard $\mathbb{V}$-collection $\Phi = \{\mathcal{C}_{ij}: i,j \in \mathbb{V}\}$, may use the Hall category to define a subalgebra 
\begin{align} \bar{H}_{\Phi} (\mathbb{V}) \subset M_{r} ( H (\mathcal{C})) \end{align}
of $r \times r$-matrices with entries in the usual motivic Hall algebra of ${\mathcal C}$. Here $r$ is the number of objects in $\mathbb{V}$. In particular, 
\begin{align} \label{hallcatalg}
\bar{H}_\Phi (\mathbb{V}) = \left\{ \left(\sum_k [\pi^k_{ij} : X^k_{ij} \to \mathcal{C}] \cdot \mathbb{L}^k \right) \in M_{r} ( H (\mathcal{C})) : \pi^k_{ij} (X^k_{ij}) \subseteq \mathcal{C}_{ij} \right\}.
\end{align}
When the $\mathbb{V}$-collection $\Phi$ is not standard, we may continue to define $\bar{H}_\Phi(\mathbb{V})$ as an algebra whose set is given in equation \eqref{hallcatalg}, but whose multiplicative structure is different. Write $\mathfrak{g}_\Phi$ for the associated Lie algebra and observe that this is a $\mathbb{V}$-graded Lie algebra. Furthermore, a pre-slicing $\mathcal{P}$ gives rise to a collection $a_{\mathcal{P}} := \left( a ( \gamma ) \right)_{\gamma \in \Gamma_\mathbb{V}}$. This is obtained as follows: for $\gamma$ corresponding to a morphism $\bar{\gamma}$ from $i$ to $j$, take $a(\gamma ) = \epsilon^{-1}_{ij} (\bar{\gamma}) \cap \mathcal{P}_{ij}$.  

\begin{definition}
A central charge $Z$ satisfies the support property relative to the pre-slicing $\mathcal{P}$ if it $(Z, a_\mathcal{P})$ satisfies the support property.
\end{definition}

The notion of a Harder-Narasimhan sequence for $\mathbb{V}$-collections contains slightly more information than the classical case.
\begin{definition} \label{defn:hn}
Let $\Phi := \{\mathcal{C}_{ij} : i, j \in \mathbb{V} \}$ be a 
$\mathbb{V}$-collection of 
 subcategories of $\mathcal{C}$, $\mathcal{P}$ a pre-slicing 
and $E \in \mathcal{C}_{ij}$ with $\epsilon_{ij} (E) = \gamma$. A 
\textit{Harder-Narasimhan sequence} for $E$ is a sequence of objects $i = i_0 , 
\ldots, i_n = j$ in $\mathbb{V}$, a sequence of objects $A_k \in 
\mathcal{C}_{i_{k }i_{k - 1}} (\theta_k)$ for which 
\begin{enumerate}
	\item $\theta_1 > \theta_2 > \cdots > \theta_k$,
	\item If $\epsilon_{i_{k}i_{k - 1}} (A_k) = \gamma_k$ then the composition,
	\begin{equation} \label{eq:sequence}
	\begin{tikzpicture}[baseline=(current  bounding  box.center), scale=1.5]
	\node (A) at (0,0) {$i = i_n$};
	\node (B) at (1,0) {$i_{n - 1}$};
	\node (C) at (2, 0) {$\cdots$};
	\node (D) at (3,0) {$i_{1}$};
	\node (E) at (4,0) {$i_0 = j$};
	\path[->,font=\scriptsize]
	(A) edge node[above] {$\gamma_n$} (B)
	(B) edge node[above] {$\gamma_{n - 1}$} (C)
	(C) edge node[above] {$\gamma_{2}$} (D)
	(D) edge node[above] {$\gamma_1$} (E);
	\end{tikzpicture} 
	\end{equation}	
	is $\gamma=\gamma_1+...+\gamma_k$.
	\item There are maps $\delta_k \in \Ext^1 (A_k, E_{k - 1}) \cap \cores_{i_{0} i_{k -1} i_{k }}$ yielding the convolution
	\begin{equation} \label{eq:sequence}
	\begin{tikzpicture}[baseline=(current  bounding  box.center), scale=1.7]
	\node (A) at (-.4,0) {$0 = $};
	\node (B) at (0,0) {$E_0$};
	\node (C) at (1, 0) {$E_1$};
	\node (D) at (2,0) {$E_2$};
	\node (Z) at (3, 0) {$\ldots$};
	\node (E) at (4,0) {$E_{n - 1}$};
	\node (F) at (5,0) {$E_{n}$};
	\node (G) at (5.4, 0) {$= E$};
	\node (C1) at (0.5, -.6) {$A_1$};
	\node (D1) at (1.5,-.6) {$A_2$};
	\node (F1) at (4.5, -.6) {$A_{n}$};
	\path[->,font=\scriptsize]
	(B) edge node[above] {} (C)
	(C) edge node[above] {} (D)
	(D) edge node[above] {} (Z)
	(Z) edge node[above] {} (E)
	(E) edge node[above] {} (F)
	(C) edge node[above] {} (C1)
	(D) edge node[above] {} (D1)
	(F) edge node[above] {} (F1);
	\path[dashed,->,font=\scriptsize]
	(C1) edge node[below=6pt, left] {$\delta_1$} (B)
	(D1) edge node[below=6pt, left] {$\delta_2$} (C)
	(F1) edge node[below=6pt, left] {$\delta_n$} (E);
	\end{tikzpicture} 
	\end{equation}		
\end{enumerate}
Given a Harder-Narasimhan sequence as above, we call the sequence $i = i_0 , 
\ldots, i_n = j$ the \textit{path} 
of the sequence.
\end{definition}
We note that in the $\mathbb{V}$-collection context, this condition translates to the existence of a special type of the usual Harder-Narasimhan sequence for $E$ in terms of semi-stable objects lying over the morphisms of $\mathbb{V}$. Now we may define a stability condition for the $\mathbb{V}$-collection.
\begin{definition} \label{defn:stabcond}
 Given a $\mathbb{V}$-collection of ind-constructible subcategories $\Phi := 
\{\mathcal{C}_{ij} : i, j \in \mathbb{V} \}$ of $\mathcal{C}$, a 
$\mathbb{V}$-stability condition $(Z, \mathcal{P})$ consists of a 
$\mathbb{V}$-central charge $Z$ and a pre-slicing $\mathcal{P}$ such that
 \begin{enumerate}
 	\item $Z$ satisfies the support property relative to $\mathcal{P}$,
 	\item $E \in \mathcal{P}_{ij} (\theta )$ then $\arg ( Z (\phi_{ij}([E])) ) \equiv \theta \pmod{2\pi} $,
 	\item every non-zero $E \in \mathcal{C}_{ij}$ has a Harder-Narasimhan sequence.
 \end{enumerate} 
\end{definition}
Observe that if $\mathbb{V}$ is the category with one object and one morphism 
and 
$\mathcal{C}_{ii} = \mathcal{C}$ along with the standard weight $\nu$, then a 
$\mathbb{V}$-stability condition reduces to the generalization of the 
Bridgeland stability conditions defined in \cite{ks08} (the 
constructibility and the important Support Property were added to the 
axiomatics of \cite{bridgeland}). Using the topology on stability data for $\mathfrak{g}_\Phi$, we obtain a topology on the space of $\mathbb{V}$-stability conditions. We also retain a version of uniqueness of 
Harder-Narasimhan sequences in this  context.
\begin{proposition} \label{prop:uniquehn}
Given a $\mathbb{V}$-stability condition on $\Phi$, $E \in \mathcal{C}_{ij}$ and a sequence $i = i_0 , \ldots, i_n = j$ of objects in $\mathbb{V}$, there is at most one Harder-Narasimhan sequence $A_1, \ldots, A_n$ of $E$ with $A_k \in \mathcal{C}_{i_{k -1 } i_k}$. Furthermore, if $B_1, \ldots, B_m$ is another Harder-Narasimhan sequence with $B_k \in \mathcal{C}_{\tilde{i}_{k - 1} \tilde{i}_k}$ then $m = n$ and $B_k = A_k$.
\end{proposition}
\begin{proof}
	This follows from the uniqueness argument of Harder-Narasimhan sequences for slicings as in \cite[Section~3]{bridgeland} applied to the pre-slicing $\mathcal{P}_{ij}$.
\end{proof}
Proposition \ref{prop:uniquehn} allows us to define the following constants.
\begin{definition}
	Given a $\mathbb{V}$-stability condition $(Z, \mathcal{P})$ on $\Phi$, $E \in \mathcal{C}_{ij}$, let $\theta^- (E) = \theta_1$ and $\theta^+ (E) = \theta_n$ if there exists a Harder-Narasimhan sequence $A_1, \ldots, A_n$ of $E$ such that $A_1 \in \mathcal{P}_{ii_1} (\theta_1)$ and $A_n \in \mathcal{P}_{i_{n - 1} j} (\theta_n)$.
\end{definition}

\subsection{Motivic 2d-4d wall-crossing formulas}

Let us now describe the wall-crossing formulas associated with the standard 
$\mathbb{V}$-stability collection.

Let $I = [\theta_1, \theta_2)$ be any half open interval in $\mathbb{R}$. 
For any sequence of objects $\mathbf{s} = (i_0, \ldots, i_r)$ in $\mathbb{V}$, 
define the ind-constructible category 
\begin{align*}
\mathcal{P}_{(i_0, \ldots, i_r)} (I) := \left\{ E \in \mathcal{C}_{i_0i_r} : \theta^\pm (E) \in I, E\text{ has a HN sequence with path }\mathbf{s} \right\}.
\end{align*}
Let $\mathcal{S}_{ij} = \cup_{n \in \mathbb{N}} \{i\} \times \mathbb{V}^n \times \{j\}$ be the set of all paths from $i$ to $j$. Using Proposition~\ref{prop:uniquehn} once again along with the finiteness of $\mathbb{V}$, we may define the element 
\begin{align}
A_{ij} (I) = \sum_{\mathbf{s} \in \mathcal{S}_{ij}} [\mathcal{P}_{\mathbf{s}} (I) \to \mathcal{C}_{ij}].
\end{align}
for every $i,j \in \mathbb{V}$. Clearly $A_{ij} (I)$ is an element of the motivic Hall algebra of $\mathcal{C}_{ij}$. Assembling these elements into a matrix gives the element
\begin{align}
A_I = \left( A_{ij} (I) \right) \in \bar{H}_\Phi (\mathbb{V}).
\end{align}
This brings us to the motivic Hall category wall-crossing formula.
\begin{theorem} \label{thm:wcfhall}
 Given a standard $\mathbb{V}$-collection of subcategories $\Phi := \{\mathcal{C}_{ij} : i, j \in \mathbb{V} \}$ of $\mathcal{C}$ and a $\mathbb{V}$-stability condition $(Z, \mathcal{P})$, if  $\theta_1 < \theta_2 < \theta_3$ then
 \begin{align}
 A_{[\theta_1, \theta_2)} \cdot A_{[\theta_2, \theta_3)} = A_{[\theta_1, \theta_3)}.
 \end{align}
\end{theorem}
The proof of this theorem follows at once from the  motivic 
wall-crossing formula \cite{ks08} along a concatenated path, 
and then summed over all paths.

Let now $\mathcal{C}$ be an ind-constructible locally Artin 
polarized $3$-dimensional Calabi-Yau category endowed with constructible 
stability structure (see \cite{ks08} for the terminology). It is explained in 
the loc.cit. that in this case for each strict sector $V$ of the plane there is 
a homomorphism $\phi$ from the motivic Hall algebra of the category 
$\mathcal{C}_V$ generated by semistables with the central charge in $V$ to the 
completed motivic quantum torus $\mathcal{R}_V$. The latter is an algebra over 
the commutative ring $Mot(pt)$ of motivic functions on the point, with added 
inverse to the motives of all groups $GL(n), n\ge 1$. Polarization means a
constructible homomorphism of abelian groups $cl: K_0(\mathcal{C})\to \Gamma$, 
where $\Gamma$ is a free abelian group endowed with skew-symmetric integer 
bilinear form $\langle,\rangle$ and the homomorphism $cl$ is compatible with 
this form and the Euler form on $K_0(\mathcal{C})$. The algebra
$\mathcal{R}_V$ is generated by the generators $e_\gamma$ such that 
$Z(\gamma)\in V, Q(\gamma)\ge 0$ (see Support Property for the 
discussion of the quadratic form $Q$) subject to the relations
$$e_{\gamma_1}e_{\gamma_1}=\mathbb{L}^{{{1}\over{2}}\langle 
\gamma_1,\gamma_2\rangle}e_{\gamma_1+\gamma_2},$$
where $\mathbb{L}$ is the motive of the affine line.
The strict sector $V$ determines the interval $I$ as above.
Applying homomorphism $\phi$ to the above $2d-4d$ wall-crossing formula,
we obtain the same factorization in the matrix algebra 
$M_r(\mathcal{R}_V):=M_r(\mathcal{R}_I)$
with entries which are elements of the motivic quantum torus.
This formula is called {\it motivic $2d-4d$ wall-crossing formula}.
Similarly to the \cite{ks08} it is equivalent to the factorization of $A_I$ 
into the clockwise product of $A_l$, where $l\subset V$ is a ray, and the 
corresponding element $A_l$ is the sum of motivic classes of all 
$\mathbb{V}$-semistable objects with the central charge belonging to $l$.
In order to recover the wall-crossing formula for numerical DT-invariants
one should assume the ``absence of poles conjecture'' from the loc. cit. 

For a geometrically minded reader and in order to make a link with 
\cite{gmn12} let us provide  a more geometric intuitive 
picture  of the numerical Donaldson-Thomas invariants in the categorical 
framework.

For that we assume that, given a generic stability condition 
$(Z, \mathcal{P})$ on the ind-constructible $A_\infty$-triangulated category 
$H^0 (\mathcal{C})$, and a class $\lambda \in K_0 (\mathcal{C})$ there is a 
moduli stack $\mathcal{M}^{ss}_\lambda$ of isomorphism classes of semi-stable 
objects with $K_0$ class $\lambda$ which is Artin and of finite type. We assume 
that there are sub-stacks $\mathcal{M}_\lambda^{ss} (0)$ which carry a virtual 
fundamental class of virtual dimension zero. The associated function 
$\Omega_{(Z, \mathcal{P})}: K_0 (\mathcal{C}) \to \mathbb{Z}$ given by 
$\Omega_{(Z, \mathcal{P})} (\lambda) = \# \mathcal{M}_\lambda^{ss} (0)$ will be 
called the DT invariant of $(Z, \mathcal{P})$.  Here $\Omega_{(Z, \mathcal{P})} 
(\lambda)$ is meant to count the ``number of semi-stable objects whose $K_0$ 
class 
equals $\lambda$''. We do not discuss here the actual meaning of these 
words, since for now, questions of when and how such a 
function is defined 
will be irrelevant. We will say that a sub-category $i: \mathcal{B} 
\hookrightarrow \mathcal{C}$ is algebraic if the moduli stack 
$\mathcal{M}^{ss}_\lambda$ in $\mathcal{B}$ is an algebraic substack of 
$\mathcal{M}^{ss}_{i_* \lambda}$. For full algebraic subcategories $\mathcal{B} 
\subset \mathcal{C}$, one may restrict the count of $\mathcal{M}^{ss}_\lambda$ 
to the substack of objects in $\mathcal{B}$. We write this function as 
$\Omega_{(Z,\mathcal{P})}^{\mathcal{B}}$.

Now assume that  a generic stability condition $(Z, \mathcal{P})$ on 
$\mathcal{C}$ 
has been chosen. Given an algebraic $\mathbb{V}$-collection of subcategories 
$\{\mathcal{C}_{ij}: i, j \in \mathbb{V}\}$ of $\mathcal{C}$, we define a 2d-4d 
wall-crossing data as follows. Let $\Gamma = K_0 (\mathcal{C}) \otimes 
\mathbb{Z}$ and the pairing $\left<, \right>$ equal to the anti-symmetrization 
of the Euler form on $\Gamma$. Take $\Gamma_{ij} = K_0 (\mathcal{C}_{ij}) 
\otimes \mathbb{Z}$ with the same pairing. Note that in our setting, we need not 
have that $\Gamma_{ij}$ is a torsor over $\Gamma$. Nonetheless, we may restrict 
$Z : \Gamma \to \mathbb{C}$ to $\Gamma_{ij}$ to obtain a central charge $Z : 
\mathbb{V} \to \mathbb{C}$. Take $\Omega = \Omega^{\mathcal{C}}_{(Z, 
\mathcal{P})}$ and $\mu : \Gamma_{ij} \to \mathbb{Z}$ as $\mu = 
\Omega^{\mathcal{C}_{ij}}_{(Z, \mathcal{P})} \circ \phi_{ij}^{-1}$. Finally, for 
$\sigma$ introduced in Section~\ref{sec:vcat}, we take $\sigma ([B], [C]) = 
(-1)^{\left<B, C \right>}$. This completes the description of numerical 
DT-invariants. 

\section{Examples from algebra} We will now give examples of motivic correspondences, enhancements and $\mathbb{V}$-collections in the algebraic setting.  We will explore $\mathbb{V}$-stability conditions arising in an elementary algebraic case.

\subsection{Algebraic motivic enhancements} The first set of examples explores a categorification for a pre-triangulated dg-category $\mathcal{A}$. 
\begin{example} \label{ex:mainalgebra}
	Let $\mathcal{A}$ be the category of twisted complexes over a given dg-category $\mathcal{C}$ and let $\mathbb{V}$ be a trivial connected groupoid, i.e. every pair of objects in $\mathbb{V}$ has a unique invertible morphism between them. Assume the  set of objects in $\mathbb{V}$ is labeled by a set $\mathcal{V}$ of objects in $\mathcal{C}$.  Given any $X,Y \in \mathbb{V}$, let $\mathcal{A}_{XY}$ be the full subcategory of twisted complexes of the form  $C_g := (Y \oplus X[-1], d_g)$ where 
	\begin{align}
	 d_g = \left[ \begin{matrix} 0 & 0 \\ g & 0 \end{matrix}
	 \right]
	\end{align}
	for some zero degree cocycle $g: Y \to X$. In other words, $\mathcal{A}_{XY}$ are all cones of morphisms from $Y$ to $X$ (shifted by $-1$). Noting that $[\mathcal{A}_{XY}] = \{[Y] - [X]\}$ and $\Gamma_{XY}$ are both one element sets, the maps $\phi_{XY} : [\mathcal{A}_{XY}] \to \Gamma_{XY}$ are uniquely determined.
	
	Given objects $C_g$ and $C_h$ lying over two composable morphisms
	\begin{align}
	X \stackrel{C_g}{\longrightarrow} Y \stackrel{C_h}{\longrightarrow} Z
	\end{align}
	Take $f_{C_g C_h} : C_g \to C_h$ to be the map given by 
	\begin{equation} \label{eq:birmap}
	\begin{tikzpicture}[baseline=(current  bounding  box.center), scale=1.5]
	\node (A) at (0,0) {$X[-1]$};
	\node (B) at (0,-1) {$Y$};
	\node (A1) at (0, -.5) {$\oplus$};
	\node (C) at (1,0) {$Y[-1]$};
	\node (D) at (1,-1) {$Z$};
	\node (C1) at (1, -.5) {$\oplus$};
	\path[->,font=\scriptsize]
	(B) edge node[below=6pt,right=.1pt] {$=$} (C);
	\end{tikzpicture} 
	\end{equation}	
It is clear that $\cone{f_{C_gC_h}}$ is quasi-isomorphic to $C_{gh} \in \mathcal{A}_{XZ}$. Associativity is also immediate from this observation.
\end{example}
Our next example is a generalization of the last one to non-trivial groupoids.

\begin{example} \label{ex:mainsubcat}
	As in the previous example, let $\mathcal{A}$ be the category of twisted complexes over a given dg-category $\mathcal{C}$ with a distinguished set of objects $\mathcal{V}$ in $\mathcal{C}$. Take $\mathcal{B} \subseteq \mathcal{A}$ to be a full pre-triangulated subcategory, $\Gamma$ its Grothendieck group $K_0 (\mathcal{B})$. and $\tilde{\Gamma}$ the Grothendieck group $K_0 (\mathcal{A})$ of $\mathcal{A}$. Define a formal element $o$ and take $\mathbb{V}$ to be the groupoid with objects labeled by $\mathcal{V} \sqcup \{o\}$. Map the additional element $o$ to $\tilde{\Gamma}$ by taking $[o] = 0 \in \tilde{\Gamma}$. For every pair $X, Y \in \mathcal{V} \sqcup \{o\}$ define the set of morphisms in $\mathbb{V}$ from $X$ to $Y$ to be the coset 
	\begin{align}
	\Gamma_{XY} = [Y] - [X] + \Gamma \subset \tilde{\Gamma}.
	\end{align}
	Composition is simply addition in $\tilde{\Gamma}$.
	
	For any $X \in \mathcal{V}$, let $\mathcal{B}_X$ be the full subcategory of objects which are quasi-equivalent to cones $\cone{f}[-1]$ or $\cone{f^\prime}$ where $f : X \to B$ or $f^\prime : B \to X$ for some $B \in \mathcal{B}$. Also take $\mathcal{B}_o := \mathcal{B}$. Then, for any $X, Y \in \mathcal{V} \sqcup \{o\}$, take $\mathcal{A}_{XY}$ to be the full subcategory of all objects $C_g := \cone{g}[-1]$ for morphisms $g : \tilde{Y} \to \tilde{X}$ where $\tilde{Y} \in \mathcal{B}_Y$ and $\tilde{X} \in \mathcal{B}_X$. Noting that any object $\tilde{X} \in \mathcal{B}_X$ has $K_0 (\mathcal{A})$ class lying in the coset $[X] + \Gamma$, and similarly with $\tilde{Y}$, we have that $[C_g] = [\tilde{Y}] - [\tilde{X}] \in [Y] - [X] + \Gamma$ yielding a well defined map $\phi_{XY} : [\mathcal{A}_{XY}] \to \Gamma_{XY}$. The morphism $f_{C_g C_h}$ is as in Example~\ref{ex:mainalgebra}. 
\end{example}
There are several variants that can be considered of examples~\ref{ex:mainalgebra} and \ref{ex:mainsubcat}. For instance, one could consider cones of any even degree morphism instead of simply zero degree morphisms. Alternatively, in Example~\ref{ex:mainsubcat}, one could restrict to shifts of cones from $\mathcal{B}$ to $X$ or to cones from $X$ to $\mathcal{B}$.  We now examine a $\mathbb{V}$-subcollection of  $\mathcal{C}_\mathcal{V}$-collection as in Example~\ref{ex:mainalgebra}.

\begin{figure}[h]
	\begin{equation*}
	\begin{tikzpicture}[baseline=(current  bounding  box.center), scale=1.8]
	\node (A) at (0,0) {$0$};
	\node (B) at (2,0) {$1$};
	\node (C) at (4,0) {$2$};
	\path[->,font=\scriptsize]
	(B) edge[bend right=10] node[above] {$M_{01}[2k + 1]$} (A)
	(C) edge[bend right=10] node[above] {$M_{12}[2k + 1]$} (B)
	(A) edge[bend right=32] node[below] {$M_{02}[2k]$} (C)
	(C) edge[bend right=32] node[above] {$M_{02}[2k + 1]$} (A)
	(B) edge[bend right=10] node[below] {$M_{12}[2k]$} (C)
	(A) edge[bend right=10] node[below] {$M_{01}[2k]$} (B);
	\end{tikzpicture}
	\end{equation*}
	\caption{\label{fig:A2i} Indecomposables in $\mathcal{C}_{ij}$}
\end{figure}
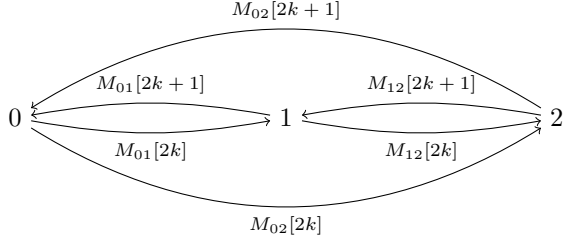

\begin{example} \label{ex:An4} Consider the $A_n$-category as in Examples \ref{ex:An1} and \ref{ex:An2}. Take $\mathbb{V} = \mathbb{Z}^{n + 1}_{\phi, A}$ to be the trivial connected quiver with objects $\{0, 1, \ldots, n\}$ and $\{\mathcal{C}_{ij}\}$ the induced $\mathbb{V}$-collection of subcategories given in  \eqref{eq:mot2cat}. In particular, take the standard motivic correspondence induced by $\phi$ and $A$ as in Definition~\ref{defn:phimotext}. 

Now let $\tilde{Z} : \mathbb{Z}^{n + 1} \to \mathbb{C}$ be a homomorphism so that, by composing with $\phi$, we obtain the central charge $Z = \phi \circ \tilde{Z}$ which, in this case, simply assigns the unique morphism $i \to j$ to $\lambda_{ij} := \tilde{Z} (e_j - e_i)$. Assume no three points in $\{\tilde{Z} (e_i)\}$ are collinear. As we observe in the next section, for $n = 2$, the space of stability conditions can be described concretely.
\end{example}

\subsection{$\Phi$-stability conditions for $A_2$-representations}
Next we consider the elementary, but rich, example of a $\mathbb{V}$-stability condition for the category of $A_2$-representations. 
	
	\label{ex:An3} Consider the $A_2$-category $\mathcal{A}$ as in Examples \ref{ex:An1} and \ref{ex:An2}. Then $\mathbb{V}$ is a trivial groupoid with objects $\{0,1,2\}$. Restricting to the indecomposable objects leaves us with the representation illustrated in Figure~\ref{fig:A2i} of the motivic  enhancement over $\mathbb{V}$. The labels over the morphisms in $\mathbb{V}$ indicate all objects lying in $\mathcal{C}_{ij}$ which are even or odd translations of the indicated indecomposable representations. 
Any function $f : \{0,1,2\} \to \mathbb{C}$ induces a homomorphism $\tilde{Z} : \mathbb{R}^3 \to \mathbb{C}$ which, in turn, yields the central charge $Z_f : K_0 (\mathcal{A}) \to \mathbb{C}$ induced by $\phi$ as in \eqref{eq:inducedcharge}. Note that $Z_f$ is uniquely determined up to a translation of $f$. For $Z_f$ to be generic, we must have that $f(i) \ne f(j)$ for all $i \ne j$. 
Let 
\[P := \frac{\{f : f(0), f(1), f(2) \text{ distinct}\}}{\text{translations}}  \] 
be the set of such central charges. One notes that $P$ is isomorphic to $P_1 \times \mathbb{C}^*$ where $P_1$ is a pair of pants (i.e. $\mathbb{C} \backslash \{0, 1\}$). To obtain a pre-slicing, we must first choose an argument $\theta_{ij}$ of $f(j) - f(i)$ for $j > i$. Such a choice yields an abelian cover $\tilde{P}$ of $P$ as well as a map $\Theta : \tilde{P} \to \mathbb{R}^3$. For a fixed central charge, the determination of a stability condition is purely dependent on the choice of the arguments (but not necessarily determined uniquely), so it remains to find the possible pre-slicings for any given element $(\theta_{01}, \theta_{12}, \theta_{02})  \in \Theta (\tilde{P})$. 

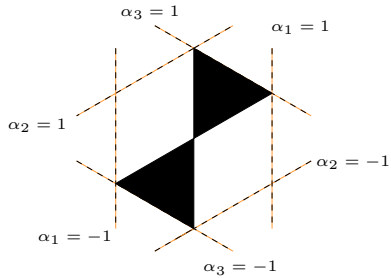
\begin{figure}[b]
	\begin{equation*}
	\begin{tikzpicture}[cross line/.style={preaction={draw=white, -, line width=6pt}}]
	\node[anchor=south west,inner sep=0] (image) at (0,0) {\includegraphics[scale=.6]{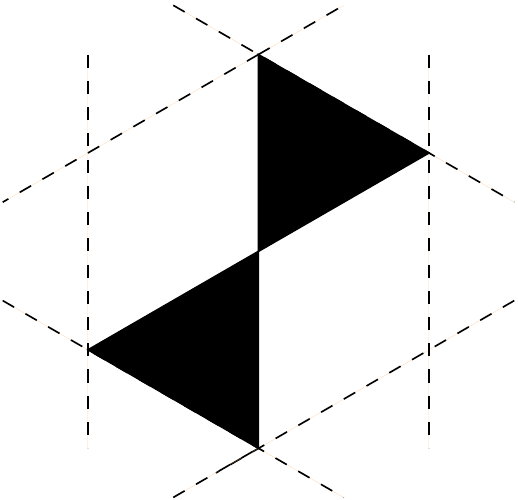}};
	\node[align=center] at (3,3) {\tiny $\alpha_1 = 1$};
	\node[align=center] at (0,.2) {\tiny $\alpha_1 = -1$};
	\node[align=center] at (1,3.2) {\tiny $\alpha_3 = 1$};
	\node[align=center] at (2.2,-.2) {\tiny $\alpha_3 = -1$};
	\node[align=center] at (3.7,1.2) {\tiny $\alpha_2 = -1$};
	\node[align=center] at (-.5,1.7) {\tiny $\alpha_2 = 1$};
	\end{tikzpicture} 
	\end{equation*}
	\caption{\label{fig:A2iii}Regions in the projection of the cube in $\Theta (\tilde{P})$.}
\end{figure}
Note that rotating $f$ by $e^{2\pi i \theta}$ results in adding $(\theta, \theta, \theta)$ to $(\theta_{01}, \theta_{12}, \theta_{02})$ and results in no effect on the constraints in the pre-slicing (nor, as we will show in a moment, the existence of Harder-Narasimhan sequences). Thus we may find $\Theta (\tilde{P})$ by considering its intersection with (or projection to) the plane
\[ V = \{(\theta_{01}, \theta_{12}, \theta_{02}) : \theta_{01} + \theta_{12} + \theta_{02} = 0 \}. \] 
Take $\alpha_1 = \theta_{02} - \theta_{01} $, $\alpha_2 = \theta_{12} - \theta_{02}$ and $\alpha_{3} = \theta_{12} - \theta_{01}$ to be elements of $V$ and note that these form the positive roots of the $A_2$-root system. 
To describe possible stability conditions, we first identify the regions in the plane $V$ where there exists an $f$ yielding $(\theta_{01}, \theta_{12}, \theta_{02})$. First note that in the ambient $\mathbb{R}^3$, we may restrict to the region $[-1,1]^3$ (as translating by an even integral lattice element results in a distinct choice of arguments for the same $f$) and that, upon projecting to $V$, we obtain a hexagonal region bounded by the inequalities $ | \alpha_i | \leq 1$ for all $i$. The existence of an $f$ then implies that either $\alpha_1 > 0$ and $\alpha_2 > 0$ or $\alpha_1 < 0$ and $\alpha_2 < 0$. These regions are depicted in Figure~\ref{fig:A2iii}.

\begin{figure}[t]
	\begin{equation*}
	\begin{tikzpicture}[baseline=(current  bounding  box.center), scale=1]
	\node (R1) at (-2,4) {Case I)};
	\node (R2) at (-2,2) {Case II)};
	\node (R3) at (-2,0) {Case III)};
	\node (A1) at (0,4) {$0$};
	\node (B1) at (2,4) {$1$};
	\node (C1) at (4,4) {$2$};
	\node (Z1) at (6,4) {$\Rightarrow$};
	\node (D1) at (8,4) {$\alpha_1 \geq 0$};
	\node (A2) at (0,2) {$0$};
	\node (B2) at (2,2) {$1$};
	\node (C2) at (4,2) {$2$};
	\node (Z2) at (6,2) {$\Rightarrow$};
	\node (D2) at (8,2) {$\alpha_2 \geq 0$};
	\node (A3) at (0,0) {$0$};
	\node (B3) at (2,0) {$1$};
	\node (C3) at (4,0) {$2$};
	\node (Z3) at (6,0) {$\Rightarrow$};
	\node (D3) at (8,0) {$1 \geq \alpha_3$};
	\path[->,font=\scriptsize]
	(B1) edge[bend right=10] node[above] {} (A1)
	(A1) edge[bend right=32] node[below] {} (C1)
	(C1) edge[bend right=32] node[above] {} (A1)
	(A1) edge[bend right=10] node[below] {} (B1)
	(C2) edge[bend right=10] node[above] {} (B2)
	(A2) edge[bend right=32] node[below] {} (C2)
	(C2) edge[bend right=32] node[above] {} (A2)
	(B2) edge[bend right=10] node[below] {} (C2)
	(B3) edge[bend right=10] node[above] {} (A3)
	(C3) edge[bend right=10] node[above] {} (B3)
	(B3) edge[bend right=10] node[below] {} (C3)
	(A3) edge[bend right=10] node[below] {} (B3);
	\end{tikzpicture}
	\end{equation*}
	\caption{\label{fig:A2ii} Inequalities for pre-slicings.}
\end{figure}
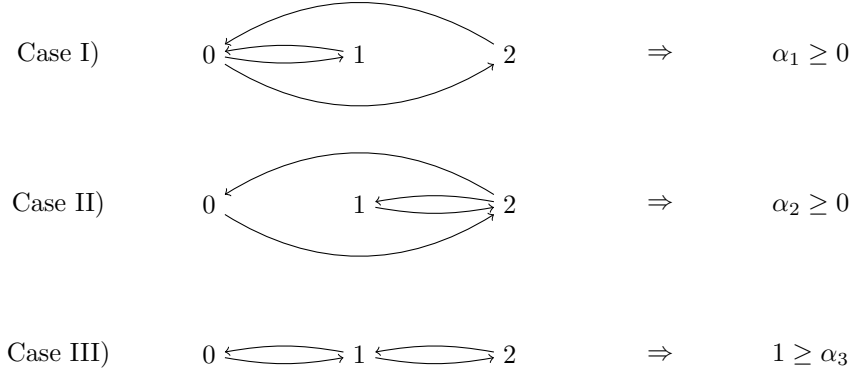

The image  is known as the coamoeba of a pair of pants when considered in $V / (\mathbb{Z}^3 \cap V)$. For a rigorous proof of this description and its relationship to the affine $A_n$ alcove decomposition, see \cite{kz, sheridan}. Thus the regions in $V$ for which there exists a $\mathbb{V}$-central charge are translates of this by $2 \mathbb{Z} \{\alpha_1, \alpha_2\}$.

\begin{figure}[h]
	\begin{equation*}
	\begin{tikzpicture}[baseline=(current  bounding  box.center), scale=1]
	\node (R1) at (0,4) {Case I)};
	\node (R2) at (4,4) {Case II)};
	\node (R3) at (8,4) {Case III)};
	\node (A1) at (1,3) {$M_{12}$};
	\node (B1) at (0,2) {$M_{01}[1]$};
	\node (C1) at (-1,3) {$M_{02}$};
	\node (A2) at (5,3) {$M_{01}$};
	\node (B2) at (4,2) {$M_{02}$};
	\node (C2) at (3,3) {$M_{12}[-1]$};
	\node (A3) at (9,3) {$M_{02}$};
	\node (B3) at (8,2) {$M_{12}$};
	\node (C3) at (7,3) {$M_{01}$};
	\node (D1) at (0,1) {$\alpha_1 > 1$};
	\node (D2) at (4,1) {$\alpha_2 > 1$};
	\node (D3) at (8,1) {$0 > \alpha_3$};
	\path[->,font=\scriptsize]
	(A1) edge node[above] {} (B1)
	(A2) edge node[above] {} (B2)
	(C1) edge node[above] {} (A1)
	(C2) edge node[above] {} (A2)
	(C3) edge node[above] {} (A3)
	(A3) edge node[above] {} (B3);
	\path[dashed, ->, font=\scriptsize]
	(B1) edge node[below] {} (C1)
	(B2) edge node[below] {} (C2)
	(B3) edge node[below] {} (C3);
	\end{tikzpicture}
	\end{equation*}
	\caption{\label{fig:ineqtrian}Inequalities associated to Harder-Narasimhan sequences.}
\end{figure}
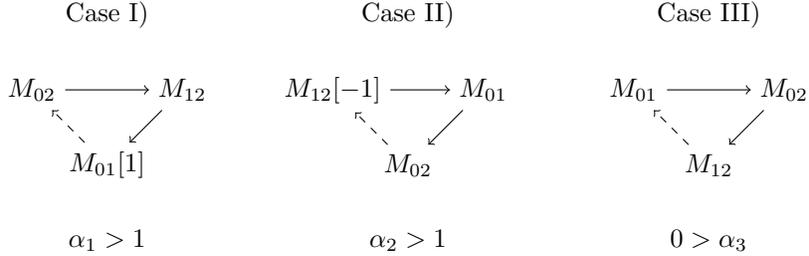

Continuing, we must now describe the stability conditions, i.e. pre-slicings which satisfy Definitions~\ref{defn:preslice}, \ref{defn:stabcond}. 
By Definition~\ref{defn:preslice:1} we must have that either all translations $M_{ij}[n]$ are semi-stable or none. Also, in order to satisfy Definition~\ref{defn:preslice:2}, equation \eqref{eq:hom} implies that the restrictions illustrated in Figure~\ref{fig:A2ii} must be made if the corresponding pairs of indicated objects are semi-stable.

Finally, we must examine when, in any of the cases from Figure~\ref{fig:A2ii}, the remaining indecomposable has a Harder-Narasimhan sequence in terms of the semi-stable indecomposables. Such sequences arise as rotations of the single exact triangle between indecomposables, each yielding a corresponding inequality illustrated in Figure~\ref{fig:ineqtrian}.

Putting these inequalities together yields a picture of the space of stability conditions $\stab{\Phi}{\mathcal{A}}$ for $A_2$ as shown in Figure~\ref{fig:A2iv}. To understand this picture, we note that there is a local homeomorphism $\pi : \stab{\Phi}{\mathcal{A}}$ to $\tilde{P}$ which (after quotienting by rotation) maps to $V$ in such a way that over the grey, yellow, blue and red regions is one to one and over the purple, green and orange regions is two to one. Over the grey central region, we must have that all indecomposables are stable.
\begin{figure}[h]
\begin{equation*}
	\begin{tikzpicture}[cross line/.style={preaction={draw=white, -, line width=6pt}}]
	\node[anchor=south west,inner sep=0] (image) at (0,0) {\includegraphics[scale=.6]{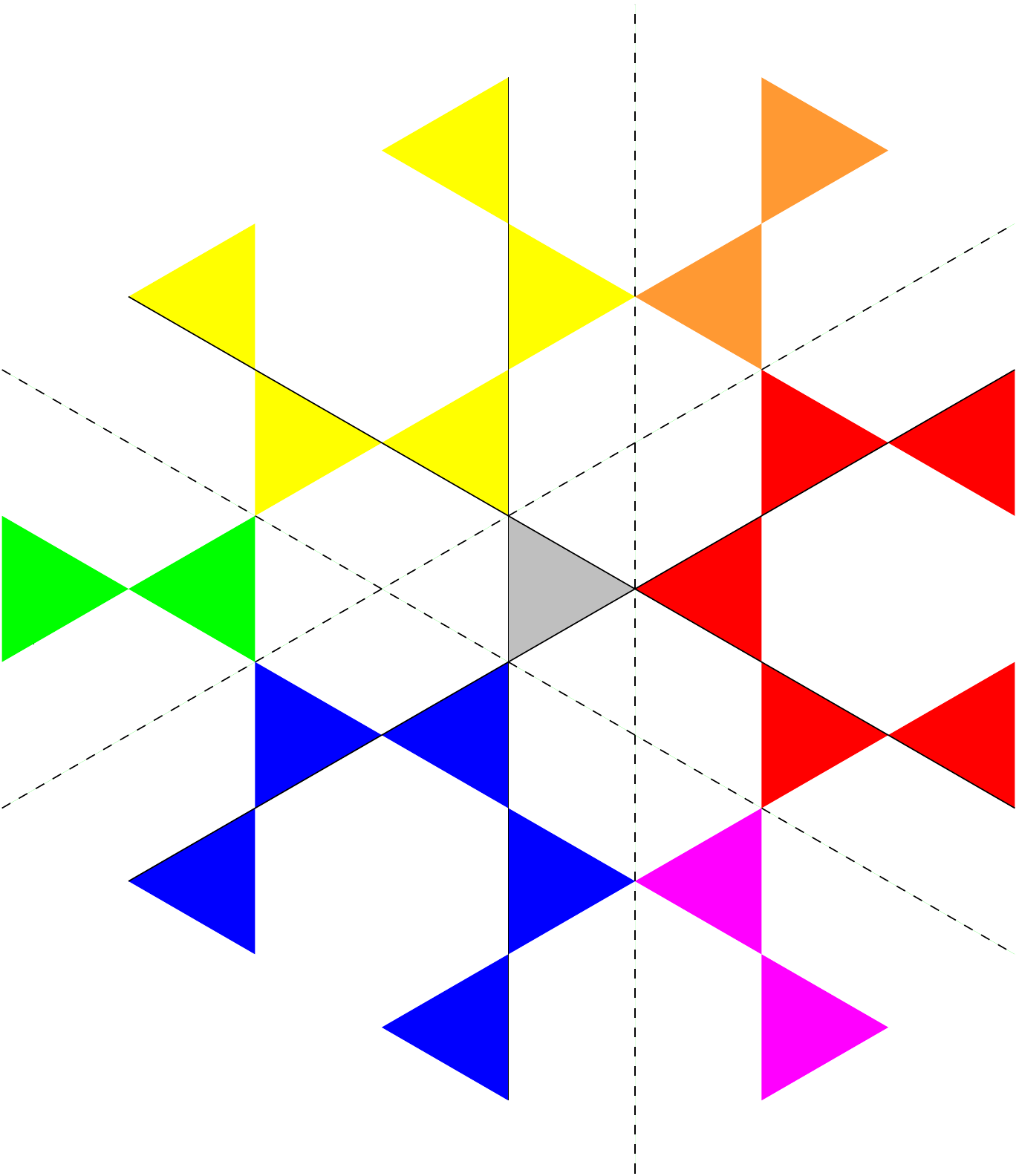}};
	\node[align=center] at (8.5,4.75) {\scriptsize \color{red} Case I};
	\node[align=center] at (2,1) {\scriptsize \color{blue} Case III};
	\node[align=center] at (2,8.5) {\scriptsize \color{yellow} Case II};
	\node[align=center] at (-.3,5.8) {\scriptsize \color{green} Cases II, III};
	\node[align=center] at (8,9) {\scriptsize \color{orange} Cases I, II};
	\node[align=center] at (8,.5) {\scriptsize \color{magenta} Case I, III};
	\end{tikzpicture} 
	\end{equation*}
	\caption{\label{fig:A2iv} Stability conditions over points in $V$.}
\end{figure}
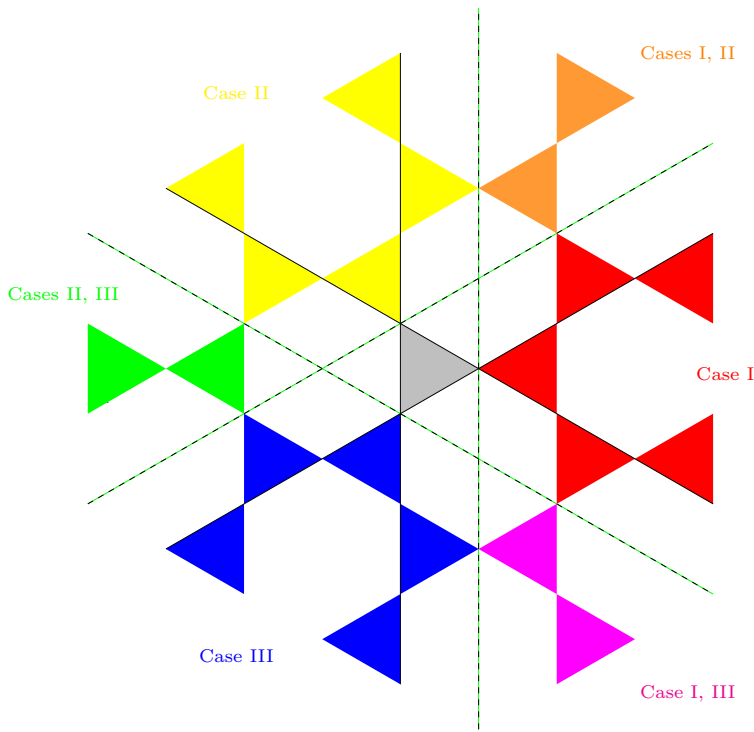
We note that the only points on $V$ over which the set of stable objects changes (locally in $\stab{\Phi}{\mathcal{A}}$) is at the vertices of the grey triangle.

The space of stability conditions for $A_n$ appears to have a similar description in terms of alcove decomposition of its affine root system. It is an intriguing question to explore analogous structures for the $D$ and $E$ root systems.

\section{Examples from symplectic geometry}
\subsection{Fukaya and matching path categories}
Let $(X, \omega)$ be a K\"ahler manifold, $Y$ a complex curve and $f : X \to Y$ a holomorphic function with complex Morse singularities. Let $\cp{f}$ and $\cv{f}$ denote the critical points and values of $f$ respectively. We assume that the critical values of two distinct critical points are distinct. For $y  \in Y$ we write $F_y$ for the fiber $f^{-1} (y)$. Let $X^\circ = X \backslash \left( \cup_{y \in \cv{f}} F_y \right)$ be the complement of the singular fibers of $f$,  $Y^\circ = Y \backslash \cv{f}$ and denote $f^\circ : X^\circ \to Y^\circ$ for the restriction of $f$.  As described in \cite{donaldson} the symplectic orthogonal to each fiber $F_y$ defines a connection on the fiber bundle $f^\circ$. Assuming reasonable behavior of $f$ at infinity (for example, one may assume that $f$ arises as the a pencil with smooth base loci and $X$ is the complement of the base loci), we then obtain a symplectic parallel transport map 
\begin{align}
\mathbf{P} : \Pi (Y^\circ) \to Symp / Ham.
\end{align}
\begin{figure}[b]
	\includegraphics[]{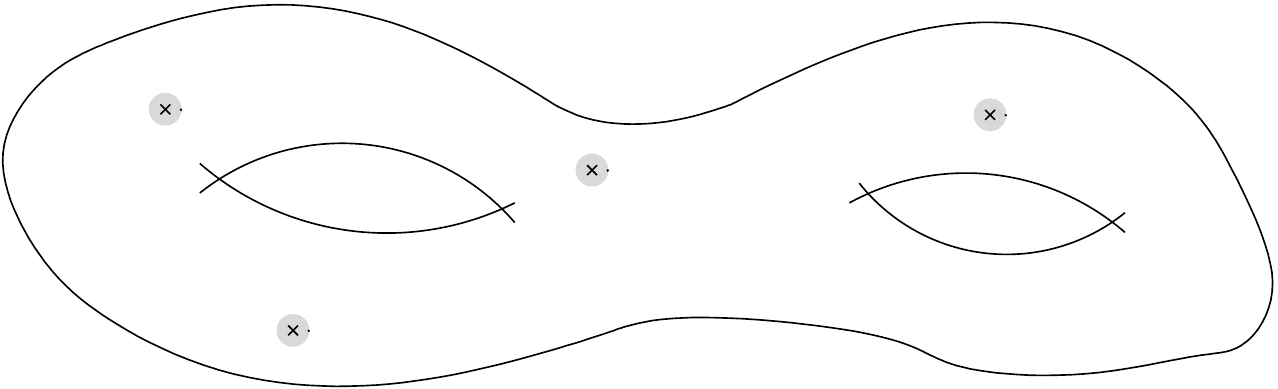}
	\caption{\label{fig:discs} $Y$ with critical points, disc neighborhoods and marked points.}
\end{figure}
This is a functor taking the fundamental groupoid $\Pi (Y^\circ )$ of $Y^\circ$ to the category of symplectic manifolds with symplectomorphisms up to Hamiltonian isotopy. Here $\mathbf{P} (y) = F_y$ and, given a path $\delta : [0,1] \to Y$ from $y_0$ to $y_1$ representing a morphism in $\Pi (Y^\circ)$, the associated symplectomorphism $\mathbf{P} (\delta )$ is obtained by symplectic parallel transport. 

\begin{definition}
	An immersed path $\delta : [0, 1] \to Y$ is called $f$-\textit{admissible} if 
	\begin{enumerate} \item $\delta ([0,1)) \subset Y^\circ$,
		\item $\delta (1) \in \cv{f}$.
	\end{enumerate}
\end{definition}

Given an $f$-admissible path $\delta$ for which $\delta (1)$ is the image $f(x)$ of the critical point $x$, we define the Lagrangian \textit{vanishing cycle} $\vc_\delta$ over $\delta$ to be the submanifold
\begin{align*}
\vc_\delta := \left\{ p \in F_{\delta (0)} : \lim_{t \to 1} \mathbf{P} (\delta|_{[0,t]}) (p) = x \right\}.
\end{align*}
We also recall that the \textit{vanishing thimble} over $\delta$ is defined as 
\begin{align*}
\vt_{\delta} = \cup_{t \in [0,1)} \vc_{\delta|_{[t,1]}} \cup \{x\}.
\end{align*}
In other words, the thimble is the union of the vanishing cycles over $\delta$ along with the critical point $x$. It is well known that $\vc_\delta$ is an exact Lagrangian sphere in $F_{\delta (0)}$ and that $\vt_\delta$ is an immersed Lagrangian ball (possibly after a suitable perturbation).

For every $y \in \cv{f}$, we fix a small embedded disc $D_y \subset Y$ so that $D_y \cap \cv{f} = y$. We also choose a point $\tilde{y} \in \partial D_y$. Then there is an unambiguous way to define the Hamiltonian isotopy class of the vanishing cycle of $y$ in $F_{\tilde{y}}$. Indeed any two embedded paths from $\tilde{y}$ to $y$ in $D_y$ yield Hamiltonian isotopic spheres. 

Now suppose $y_0, y_1 \in \cv{f}$ and $\tilde{y}_0, \tilde{y}_1$ are their neighboring points. An immersed path $\delta : [0,1] \to Y \backslash \left(\cup D_{\tilde{y}} \right)$ with $\delta (0) = \tilde{y}_0$ and $\delta (1) = \tilde{y}_1$ will be called a \textit{matching path} if $\mathbf{P} (\delta ) (\vc_0)$ is Hamiltonian isotopic to $\vc_1$. Extending such a matching path to the critical values by concatenating with $\delta_{y_0}$ and $\delta_{y_1}$, one has an immersed path in $Y$ over which lies an immersed Lagrangian sphere $\mc_\delta$ called the \textit{matching cycle} of $\delta$.

\begin{example}
Take $X,Y$ and $f$ as above. Let $\tilde{Y} \subset Y$ be a sub-surface with boundary $\partial Y \cong S^1$ and $\{y_1, \ldots, y_{r + 1}\} = \partial Y \cap \cv{f}$ the critical values of $f$ on the boundary oriented clockwise. Also assume that $\tilde{Y} = \cup_{i = 1}^r \tilde{Y}_i$ where
\begin{enumerate}
\item $y_{i - 1}, y_i \in \tilde{Y}_i$,
\item $\tilde{Y}_i \cap \tilde{Y}_j = \empty$ for $j \ne i + 1$ and,
\item $\tilde{Y}_i \cap \tilde{Y}_{i + 1}$ is a closed subset of the boundary of both.
\end{enumerate} 
Let $\mathcal{D}$ be the subcategory of the fundamental groupoid $\Pi (\tilde{Y} )$ defined as follows. The objects of $\mathbb{D}$ are  the critical values $\{y_1, \cdots, y_r\}$ of $f$ on the boundary of $\tilde{Y}$. The morphisms are defined via
\begin{align}
\Hom_{\mathcal{D}} (y_i, y_j) = \begin{cases}
\{\delta \in \Pi (\cup_{l = i}^{j - 1} \tilde{Y}_i) : \delta (0) =  y, \delta (1) =  y_j\} & \text{if } j > i, \\
\{1\} & \text{if } j = i, \\
\emptyset & \text{if } j < i.
\end{cases}
\end{align} 
In general, it could occur that the groupoid $\mathbb{V}$ may not be connected. 

Take $\mathcal{C}$ to be the Fukaya category of $X$. For $y_i, y_j \in \mathbb{V}$, we take $\mathcal{C}_{y_i y_j}$ to be the full subcategory of $\mathcal{C}$ with objects
\begin{align}\{\mc_\delta : \delta \in \Hom_{\mathbb{V}} (y_i, y_j) \}.
\end{align}

For two composable morphisms $[\delta] \in \Hom_{\mathcal{D}} (y_i, y_j)$ and $[\delta^\prime] \in \Hom_{\mathcal{D}} (y_j, y_k)$, the critical point $y_j$ is a distinguished intersection point $ \mc_\delta \cap \mc_{\delta^\prime}$. Thus for the motivic correspondence we simply take
\begin{align}
\mathcal{\cores}_{y_i y_j y_k} = \{y_j\} \times \mathcal{C}_{y_i y_j} \times \mathcal{C}_{y_j y_k}
\end{align}
along with the trivial weight $\nu = 0$.

By the result of \cite[Theorem~17.16]{seidel}, we have that $\cone{y_j}$ is a matching cycle over the matching path which is the concatenation $\delta^\prime * \delta \in \Hom_{\mathcal{D}} (y_i, y_k)$.
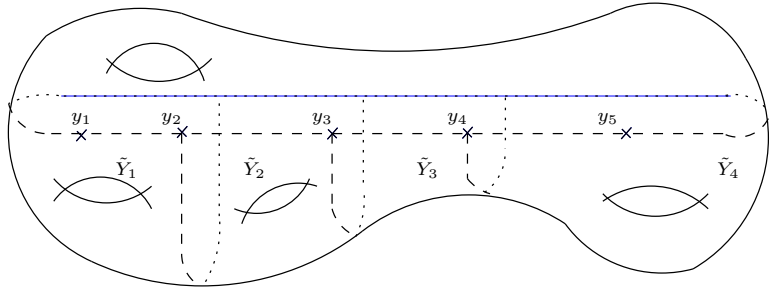
\begin{figure}[h]
	\begin{equation*}
	\begin{tikzpicture}[cross line/.style={preaction={draw=white, -, line width=6pt}}]
	\node[anchor=south west,inner sep=0] (image) at (0,0) {\includegraphics[scale=1]{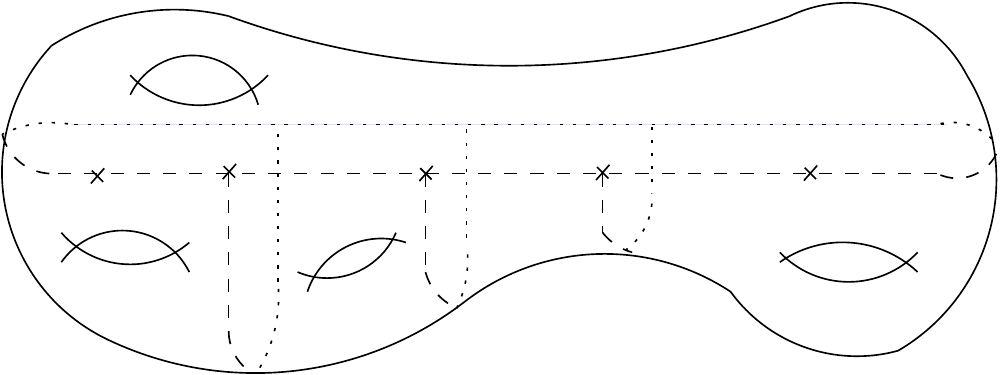}};
	\node at (9.6,1.6) {\scriptsize  $\tilde{Y}_4$};
	\node at (5.6,1.6) {\scriptsize  $\tilde{Y}_3$};
	\node at (3.3,1.6) {\scriptsize  $\tilde{Y}_2$};
	\node at (1.6,1.6) {\scriptsize  $\tilde{Y}_1$};
	\node at (8,2.25) {\scriptsize  $y_5$};
	\node at (6,2.25) {\scriptsize  $y_4$};
	\node at (4.2,2.25) {\scriptsize  $y_3$};
	\node at (2.2,2.25) {\scriptsize  $y_2$};
	\node at (1,2.25) {\scriptsize  $y_1$};
	\end{tikzpicture} 
	\end{equation*}
	\caption{\label{fig:sympex1} The regions $\tilde{Y}_i$.}
\end{figure}
\end{example}

We follow this general construction with a particular case.
\begin{example}
Take $X$ to be the exact symplectic manifold which is the smooth fiber of an $A_n$-singularity and $Y = \mathbb{C}$. For purposes of illustration, consider $X$ to be the Riemann surface $\{ (x,y) : x^2 + y^2 + x^{n + 1} = 1\}$ and take $f : X  \to \mathbb{C}$ to be a Lefschetz fibration which is a double branched cover. We may assume the critical values $\{y_0, \ldots, y_{n}\}$ of $f$ are $\{0, \ldots, n\} \subset \mathbb{R}$. Letting $y_{-1} = - \infty$ and $y_{n + 1} = \infty$, we define the regions $\tilde{Y}_i = \{z \in \mathbb{C}: \text{Im} (z) \leq 0, y_{i - 1} \leq \text{Re} (z) \leq y_i \}$ and take $\tilde{Y}$ to be the negative imaginary half-plane. Then it is clear that for $i < j$ $\Hom_{\mathcal{D}} (y_i, y_j)$ consists of one path $\delta_{ij}$, over which there is a matching cycle which is a circle in $X$. Either through a direct appeal \cite{khovanovseidel}, or through an elementary computation, one exhibits an isomorphism between the motivic correspondence in this case and that in Example~\ref{ex:An2.5}.
\end{example}

\subsection{Fukaya-Seidel categories}
Other examples that arise in symplectic geometry require additional background.  In this section we adapt a notion of the Fukaya-Seidel category to one over an arbitrary base curve $Y$ (as opposed to the conventional definition over $\mathbb{C}$). We will eschew some open questions regarding generators of this category and simply define it using a basic class of generators. The classical construction which we refer to is given in \cite{seidel}. One aim is to apply matching path techniques to give some geometrically motivated examples of 2d-4d categorified stability structures.

First we define the category. As a preliminary step, we consider a particular flow on the disc $D = \{z \in \mathbb{C} : |z| \leq 1\}$. Let $g : D \backslash \mathbb{R}_{\leq 0} \to \mathbb{R} \times \mathbb{R}_{\geq 0}$ be an orientation preserving diffeomorphism taking $re^{i\theta}$ to $(- \ln (r), \tilde{g} (\theta))$ for a monotonic $\tilde{g}$. Let $H : \mathbb{R}^2 \to \mathbb{R}_{\geq 0}$ be a function that is zero on $\mathbb{R}_{\leq 0} \times \mathbb{R}$ and, for some constant $C > 0$, equal to $x$ for $x > C$. Taking the standard symplectic form $\omega_{st}$ on $\mathbb{R}^2$, consider the time $1$ Hamiltonian flow $\tilde{\phi}_{H} : \mathbb{R} \times \mathbb{R}_{\geq 0}$ and let $\phi$ be its pullback $g^* (\tilde{\phi}_H)$. It is clear that $\phi$ is the identity on $\partial D$. Now, let $f : X \to Y$ be as before and choose a regular basepoint $z \in Y$ along with a real tangent vector $\partial_z \in T_z Y$. Consider a disc $D_z$ centered at $z$ with $- \partial_x = \partial_z$. Identifying $D_z$ with $D$, we may extend $\phi$ on $D_z$ to all of $Y$. 

To define the Fukaya-Seidel category $\fs{f}{z}$, we start with an initial $A_\infty$-category $\pfs{f}{z}$. Let $X_z = X \backslash F_z$ and define an admissible Lagrangian thimble to be an open thimble $\vt_{\delta} \subset X_z$ of an $f$-admissible path $\delta : [0, 1] \to Y$ such that $\delta (0) = z$ and $\delta^\prime (0) \ne \partial_z$. Using the quadratic holomorphic form on $X$, one may equip $\vt_\delta$ with a grading $\alpha$ and decorate it with a pin structure $\beta$. Take the triple $L_\delta = (\vt_{\delta}, \alpha, \beta)$ to be a \textit{Lagrangian brane}. 

Morphisms in $\pfs{f}{z}$ are given by Floer complexes
\begin{align*}
\Hom_{\pfs{f}{z}} (L_0, L_1) := CF^* (L_0, \phi (L_1)).
\end{align*}
The grading in this complex is determined by the respective gradings of the Lagrangian branes and the signs of the differential, composition and higher compositions are determined by pin structures. This definition assumes a universal choice of perturbation datum in order to ensure transversality. 

\begin{definition}
	The Fukaya-Seidel category $\fs{f}{z}$ of $f$ at $z$ is the category of twisted complexes $\textnormal{Tw} (\pfs{f}{z})$.
\end{definition}

When $Y = \mathbb{C}$, there is an alternative and equivalent description of $\fs{f}{z}$.  Let $\tilde{X}$ be defined as the pullback 
\begin{equation}
\begin{tikzpicture}[baseline=(current  bounding  box.center), scale=1.5]
\node (A) at (0,0) {$\tilde{X}$};
\node (B) at (1,0) {$X$};
\node (C) at (0, -1) {$\mathbb{C}$};
\node (D) at (1,-1) {$\mathbb{C}$};
\path[->,font=\scriptsize]
(A) edge node[above] {} (B)
(B) edge node[right] {$f$} (D)
(A) edge node[] {} (C)
(C) edge node[below] {$(u - z)^2$} (D);
\end{tikzpicture} 
\end{equation}
The $\mathbb{Z} / 2$-action on $\mathbb{C}$ given by taking $(u -z)$ to $- (u - z)$ lifts to a $\mathbb{Z} / 2$ symplectic action on $\tilde{X}$. This induces an action on the Fukaya category of compact Lagrangians $\mathcal{F} (\tilde{X})$. The theorem of \cite[Theorem~18.24]{seidel} gives an equivalence between the $\mathbb{Z} / 2$ invariant subcategory of a Fukaya category $\mathcal{F} (\tilde{X})^{\mathbb{Z} / 2}$. When $Y \ne \mathbb{C}$, this alternative route is not as readily available. In the first place, there is not a unique branched double cover $\tilde{Y} \to Y$ with sole ramification point $z \in Y$, and indeed, such a cover may not exist at all.

The Grothendieck group of a Fukaya-Seidel category admits a homomorphism to the relative homology 
\begin{align} \phi : K_0 (\fs{f}{z}) \to H_n ( X, f^{-1} (z)) \end{align}
which takes a Lagrangian to its homology class. Composing with the connecting homomorphism gives the map
\begin{align}
\phi^\partial : K_0 (\fs{f}{z}) \to H_{n - 1} (f^{-1} (z)).
\end{align}
Using Definition~\ref{defn:phimotext}, one may choose to study the motivic  enhancements of the groupoids $H_n ( X, f^{-1} (z))_\phi$ (respectively $H_{n - 1} (f^{-1} (z))_{\phi^\partial}$) by $\fs{f}{z}$.  While interesting, the standard $\mathbb{V}$-collection arising from this setup is difficult to access. Instead, we will reduce the dimension of our problem by considering potentials which factor through surfaces. Before discussing this approach, we consider a basic motivic  enhancement in the Fukaya-Seidel setting.

\subsection{Surface factorizations}
The next construction is inspired by the approach in \cite{gmn13} to 2d-4d wall-crossings in Hitchin systems, which will be commented on more thoroughly in section \ref{sec:hitchin}. In this case, $Y$ is the base curve $C$, $X$ is a conic bundle over cotangent bundle $T^*C$ ramified over the spectral curve $\Sigma$. As it turns out, there is a much more general framework which may be employed to provide a categorification of this setup. Given any $f : X \to Y$ we write $\tilde{\Delta}_f \subset X$ for the discriminant of $f$; namely, the set of points in $X$ for which the total derivative of $f$ does not have maximal rank. We take $\Delta_f$ to be the image $f(\tilde{\Delta}_f)$. 
	
\begin{definition} A \textit{surface factorization} $\mathcal{S} = (f,g,h, X,S,Y)$ of a holomorphic function \[f : X \to Y \] is a the commutative diagram of holomorphic functions
\begin{equation}
\begin{tikzpicture}[baseline=(current  bounding  box.center), scale=1.5]
\node (A) at (0,0) {$X$};
\node (B) at (1,0) {$S$};
\node (C) at (.5, -1) {$Y$};
	\path[->,font=\scriptsize]
	(A) edge node[above] {$g$} (B)
	(B) edge node[below right] {$h$} (C)
	(A) edge node[below left] {$f$} (C);
\end{tikzpicture} 
\end{equation}
where $S$ is a complex surface such that 
 \begin{enumerate}
 	\item  $h$ has no critical points, 
 	\item $\Delta_g$ is a smooth complex curve and $g|_{\tilde{\Delta}_g}$ is an isomorphism.
 	\item $h|_{\Delta_g}$ has at worst Morse critical points.
 \end{enumerate}
\end{definition}
This definition ensures that $f$ also has at worst Morse critical points. The symplectic version of this definition is that of a Lefschetz bifibration and can be found in \cite[Section~15e]{seidel}.

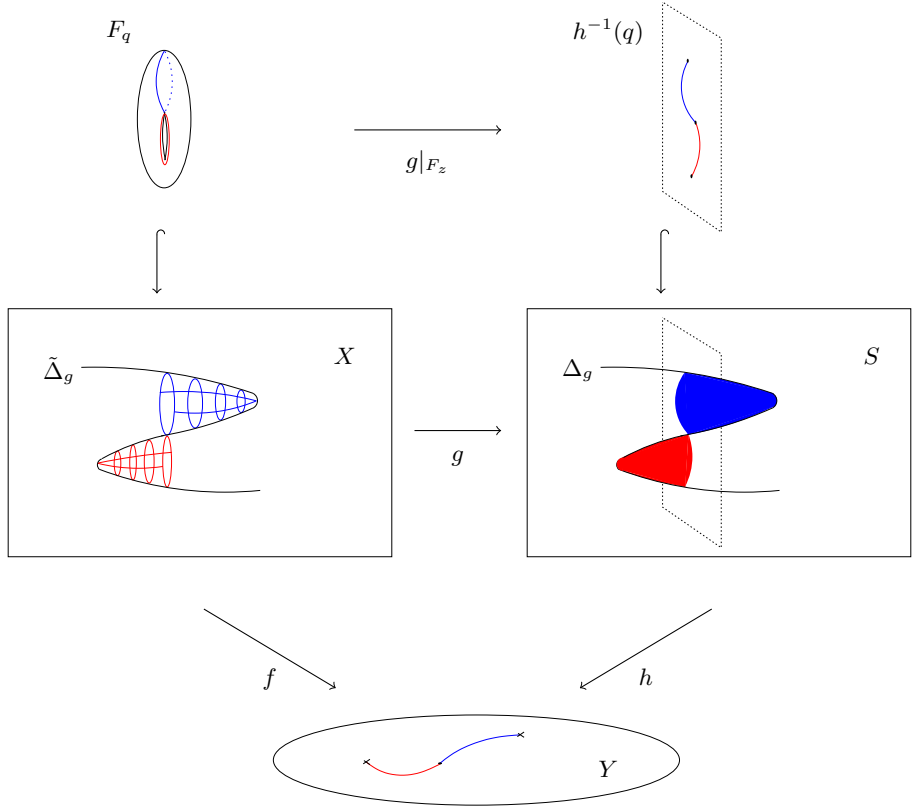
\begin{figure}[h]
	\begin{equation*}
	\begin{tikzpicture}[cross line/.style={preaction={draw=white, -, line width=6pt}}]
	\node[anchor=south west,inner sep=0] (image) at (0,0) {\includegraphics[scale=.6]{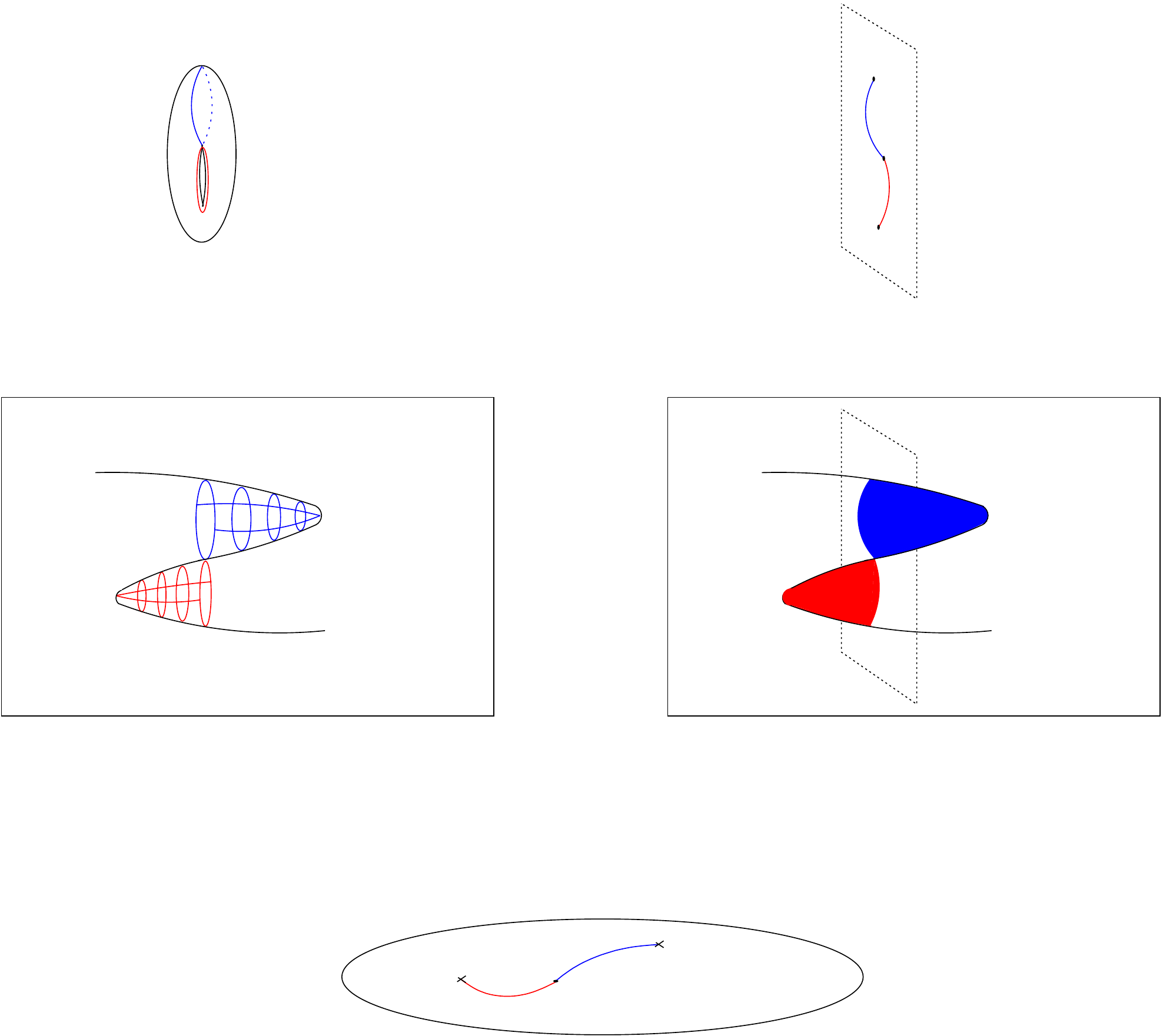}};
	\node at (4.5,6) {\small $X$};
	\node at (.7,5.8) {\small $\tilde{\Delta}_g$};
	\node at (7.6,5.8) {\small $\Delta_g$};
	\node at (11.5,6) {\small $S$};
	\node at (8,.5) {\small $Y$};
	\node at (1.5,10.3) {\small $F_q$};
	\node at (8,10.3) {\small $h^{-1} (q)$};
	\node (X1) at (5.3,5) {};
	\node (S1) at (6.7,5) {};
	\node (Y1) at (4.5,1.5) {};
	\node (Y2) at (7.5,1.5) {};
	\node (X2) at (2.5,2.7) {};
	\node (S2) at (9.5,2.7) {};
	\node (Ff) at (4.5,9) {};
	\node (Ff2) at (2,7.8) {};
	\node (X3) at (2,6.7) {};
	\node (Fh2) at (8.7,7.8) {};
	\node (S3) at (8.7,6.7) {};
	\node (Fh) at (6.7,9) {};
	\path[->,font=\small]
	(X2) edge node[below=4pt] {$f$} (Y1)
	(S2) edge node[below=4pt] {$h$} (Y2)
	(X1) edge node[below=4pt] {$g$} (S1)
	(Ff) edge node[below=4pt] {$g|_{F_z}$} (Fh);
	\path[right hook->,font=\small]
	(Ff2) edge node {} (X3)
	(Fh2) edge node {} (S3);
	\end{tikzpicture} 
	\end{equation*}
	\caption{\label{fig:sf1} Vanishing thimbles, cycles and matching paths of a surface factorization.}
\end{figure}
Given a surface factorization of $f$, an $f$-admissible path $\delta$ will also be $h|_{\Delta_g}$-admissible. Thus there are two vanishing thimbles, $\vt^f_\delta$ an exact Lagrangian ball in $X$ and an interval $\vt^h_\delta \subset \Delta_g$, associated to such a path. The thimble $\vt^f_\delta \subset X$ has another feature. Namely, for any regular value $q \in Y$, the map $g$ restricted to the fiber $F_q = f^{-1} (q)$ is a complex Morse function with values in $h^{-1} (q)$ and critical values in $\Delta_g \cap h^{-1} (q)$. Furthermore, for $q = \delta (t)$ and $t \ne 1$, the vanishing cycle $\vc^f$ of $\delta$ over $q$ is a matching cycle for a matching path $\bar{\delta}^q : [0,1] \to h^{-1} (q)$.
Letting $D^+$ be the upper half disc $\{z \in \mathbb{C}: |z| \leq 1, \text{Im} (z) \geq 0\}$, this implies that to any $f$-admissible path $\delta$ from a basepoint $z \in Y$, there is a function 
\begin{align}
u_\delta : (D^+, \partial D^+) \to (S, \Delta_g \cup h^{-1} (z))
\end{align}
which takes the interval $I_t = \{z \in D^+ : \text{Im} (z) = t\}$ to the matching path $\bar{\delta}^{\delta (t)}$. Such a disc is determined up to isotopy relative the boundary and results in a homomorphism
\begin{align} \label{eq:ktoh2}
\chi_{\mathcal{S}} : K_0 (\fs{f}{z}) \to H_2 (S, \Delta_g \cup h^{-1} (z); \mathbb{Z}).
\end{align}
This setup is illustrated in Figure~\ref{fig:sf1}.

Fixing a surface factorization $\mathcal{S}$ of $f$ and choice of a basepoint $z \in Y$ over which both $f$ and $h|_{\Delta_g}$ are regular, we define a groupoid $\mathbb{V}_\mathcal{S}$. The objects of $\mathbb{V}_\mathcal{S}$ will be defined as
\begin{align} \label{eq:defgamma}
Ob (\mathbb{V}_\mathcal{S} ) = h^{-1} (z) \cap \Delta_g.
\end{align}
Let $\Gamma \subset H_1 (\Delta_g; \mathbb{Z})$ be the subgroup generated by all matching cycles of $h|_{\Delta_g}$. 
To describe the morphisms in $\mathbb{V}_\mathcal{S}$, let $\Gamma^\partial \subset H_1 (\Delta_g, \mathbb{V}_\mathcal{S}; \mathbb{Z})$ be the subgroup generated by vanishing thimbles of $h|_{\Delta_g}$. Consider the connecting homomorphism from long exact sequence of relative homology
\begin{align*}
\delta : H_1 (\Delta_g, \mathbb{V}_\mathcal{S}; \mathbb{Z}) \to H_0 (\mathbb{V}_\mathcal{S}; \mathbb{Z}).
\end{align*}
For $i \in \mathbb{V}_\mathcal{S}$ take $\bar{i}$ to be the class in $H_0 (\mathbb{V}_\mathcal{S}; \mathbb{Z})$ and define
\begin{align*}
\Hom_{\mathbb{V}_\mathcal{S}} (i, j) := \delta^{-1} (\bar{j} - \bar{i}) \cap \Gamma^\partial .
\end{align*}
We note that $\mathbb{V}_\mathcal{S}$ is a connected groupoid if and only if $\Delta_g$ is connected. Using the language of Definition~\ref{defn:gpdphi}, $\mathbb{V}_\mathcal{S}$ is the groupoid induced by $\delta$ and $A = \{[p] : p \in h^{-1} (z) \cap \Delta_g\}$. A standard  $\mathbb{V}_\mathcal{S}$-collection associated to this groupoid will be examined in \ref{sec:stcoll}. First we will consider a more restrictive  enhancement in this context.

\subsection{Motivic  enhancements by $\fs{f}{z}$}
We now describe a collection of subcategories $\{\mathcal{C}_{ij}\}$ of $\fs{f}{z}$ lying over the category  $\mathcal{D}$. If the resulting collection yields a $\mathbb{V}_\mathcal{S}$-collection, by Definition~\ref{defn:Vcoll}, we must provide an odd motivic   enhancement $\Phi = (\{\mathcal{C}_{ij}, \epsilon_{ij}\}, \{(\cores_{ijk}, \nu_{ijk})\})$ of $\mathbb{V}_\mathcal{S}$ for which $K_\Phi (\mathbb{V}_\mathcal{S}) \cong \mathbb{V}_\mathcal{S}$. 

We begin by taking the ambient ind-constructible $A_\infty$-category to be $\fs{f}{z}$.

Given an $f$-admissible path $\delta$, recall that the vanishing cycle $\vc_\delta^f$ is Hamiltonian isotopic to a matching cycle $\mathcal{M}_{\bar{\delta}}$ over a matching path $\bar{\delta} : [0,1] \to F_z$. Note that $\bar{\delta}$ is determined up to isotopy in $F_z \backslash \mathbb{V}_\mathcal{S}$. Using this fact, we will choose tangent vectors in $T_i (F_z)$ for each $i \in \mathbb{V}_\mathcal{S}$ assume that $\bar{\delta}^\prime (0)$ and $\bar{\delta}^\prime (1)$ agree with these choices. This enables a grading of $\vt^f_\delta$ to induce one on $\vt^h_\delta$.  Such a grading is part of a brane structure on $L^f_\delta = (\vt_\delta^f, \alpha, \beta)$ (inducing one on $\vt_{\delta}^h$).  Of course, the grading on $\vt_\delta^h$ induces an orientation  yielding an oriented path in $\Delta_g$ that connects two points in $\mathcal{V}$. Write $[L^g_\delta]$ for the associated homology class in $H_1 (\Delta_g, \mathbb{V}_\mathcal{S}; \mathbb{Z})$ and define $\mathcal{C}_{ij}$ to be the full subcategory of $\fs{f}{z}$ with objects
\begin{align}
Ob (\mathcal{C}_{ij}) = \left\{ L_\delta^f : [L^g_\delta] \in \Gamma_{ij} \right\}.
\end{align}
The maps $\epsilon_{ij}$ are simply given by $\epsilon_{ij} (L_\delta^f) = [L_\delta^g]$.

Now, given $L_{\delta_1}^f \in \mathcal{C}_{ij}$ and $L_{\delta_2}^f \in \mathcal{C}_{jk}$ there is a unique intersection point $q \in L_{\delta_1} \cap L_{\delta_2}$ at $j$ whose Maslov index is difference of the orientations of $\vt_{\delta_1}^f$ and $\vt_{\delta_2}^f$ and therefore odd. We take $\cores_{ijk}$ to be the union of these intersections over such objects.   It is known that concatenation of vanishing thimbles in $\Delta_g$ corresponds to a mutation of underlying paths and that the categorical incarnation of this is to take the cone of a morphism between the two thimbles of $f$. It is now not difficult to define the functor $\phi$. In particular, $\phi_{ij}$ simply takes $[L_\gamma^f] = [\vt^h_\gamma]$. 

\begin{remark} \label{remark:wcfgen}
	While it is unclear whether the above collection of categories can be assembled into a non-standard $\mathbb{V}_\mathcal{S}$-collection $\Phi$, it is certain that any such collection will not have a $\Phi$-stability condition. This is due to the fact that such stability conditions require $\mathcal{C}_{ij}$ to be stable under even translations, which would force the zero morphism to be contained in the support of the motivic correspondence between two such categories. This suggests again that a more flexible framework of stability conditions in this context may be more useful in the non-standard setting.
\end{remark}

\subsection{Standard $\mathbb{V}_\mathcal{S}$-collections} \label{sec:stcoll}
The motivic  enhancement considered in the previous section of a surface factorization is not standard. In other words, the motivic correspondence does not contain the full $\Ext^1$ group. This does obstruct a direct application of the wall-crossing formula in Theorem~\ref{thm:wcfhall} to the symplectic case. However, an analogous formula is expected to apply in certain non-standard contexts (see Remark~\ref{remark:wcfgen}). Alternatively, we may consider a larger $\mathbb{V}_\mathcal{S}$-collection which is standard. To obtain this collection, recall that equation~\eqref{eq:ktoh2} gave the homomorphism 
\begin{align}
\chi_{\mathcal{S}} : K_0 (\fs{f}{z}) \to H_2 (S, \Delta_g \cup h^{-1} (z); \mathbb{Z}).
\end{align}
 Consider the triple $h^{-1} (z) \subset \Delta_g \cup h^{-1} (z) \subset S$ and the connecting homomorphism $\tilde{\delta}$ of the long exact sequence of the homology of this triple. Along with excision, we obtain the map
\begin{align}
\chi^\prime_{\mathcal{S}} : K_0 (\fs{f}{z}) \to H_1 (\Delta_g ,  h^{-1} (z) \cap \Delta_g; \mathbb{Z}) = H_1 (\Delta_g, \mathbb{V}_\mathcal{S} ; \mathbb{Z}).
\end{align}
As the image of $\chi_{\mathcal{S}}$ is generated by of a set of surfaces whose boundary along $\Delta_g$ is a vanishing thimble $\vt^h$, it follows that the image of $\chi^\prime_{\mathcal{S}}$ is $\Gamma$ as defined after equation \eqref{eq:defgamma}. Using this data with Definition~\ref{defn:phimotext} we obtain a standard $\mathbb{V}$-collection.

\begin{definition}
	Given a surface factorization $\mathcal{S}$, the \textit{standard $\mathbb{V}_\mathcal{S}$-collection} is the motivic  enhancement induced by $\delta, \chi^\prime_S$ and $\mathbb{V}_\mathcal{S}$.
\end{definition}

Let us examine the examples of an $A_n$ singularity in general.
\begin{example}
	The $A_n$-singularity $f: \mathbb{C}^3 \to \mathbb{C}$ is given by  \[f(x_1, x_2, x_3) = x_1^{n + 1} + x_2^2 + x_3^2 .\]
	The Fukaya-Seidel category of this singularity has been thoroughly studied in \cite{kseid}, so we will summarize briefly and put it into the context of a 2d-4d categorification. 
	To obtain Morse singularities, we choose a perturbation 
	\[ f(x_1, x_2, x_3) = x_1^{n + 1} - (n + 1)\varepsilon x_1 + x_2^2 + x_3^2 \]
	which admits the factorization $f = hg$ where $g : \mathbb{C}^3 \to \mathbb{C}^2$ via $g (x_1,x_2, x_3) = (x_1, x_2^2 + x_3^2)$ and $h : \mathbb{C}^2 \to \mathbb{C}$ via $h(u_1, u_2) = u_1^{n + 1}  - (n + 1) \varepsilon u_1 + u_2$. One obtains  
	\begin{align}
	\tilde{\Delta}_g & = \{(x, 0, 0): x \in \mathbb{C}\}, \\ 
	\Delta_g & = \{(x, 0) : x \in \mathbb{C} \} \cong \mathbb{C}.
	\end{align} 
	So $h|_{\Delta_g} : \mathbb{C} \to \mathbb{C}$ is simply the polynomial $x^{n + 1} - \varepsilon x$.  The critical points and values are then 
	\begin{align}
	\cp{f}  = \{(\zeta \sqrt[n]{\varepsilon}, 0, 0): \zeta^n = 1\} , \\ 
	\cv{f} = \cv{h|_{\Delta_g}}  = \{ -n \zeta \sqrt[n]{\varepsilon} : \zeta^{n} = 1 \}.
	\end{align}
	Choose a base point $z \not\in \cv{f}$.

	In this case, there are no matching paths for $h: \Delta_g \to \mathcal{C}$ so that $\Gamma = \{0\}$ (which is also clear from the fact that $\Gamma \subset H_1 (\mathbb{C}; \mathbb{Z})$). Thus $\mathbb{V}_\mathcal{S}$ is a connected trivial groupoid with $(n + 1) = | h|_{\Delta_g}^{-1} (z)|$ objects. Recalling the equation~\eqref{eq:phiAn} from Example~\ref{ex:An1} dealing with representations of the $A_n$-quiver $Q = (Q_0, Q_1)$ and labeling $Ob (\mathbb{V}_\mathcal{S} )= h^{-1} (z) \cap \Delta_g$ as $\{0, \ldots, n\}$, we see that there is a commutative diagram realizing the connecting homomorphism of the homology sequence with $\phi$
	\begin{equation*}
	\begin{tikzpicture}[baseline=(current  bounding  box.center), scale=1.5]
	\node (A) at (.5,1) {$\hbox{}$};
	\node (B) at (2,1) {$H_1 (\Delta_g, \mathbb{V}_\mathcal{S}; \mathbb{Z})$};
	\node (C) at (4, 1) {$H_0 (\mathbb{V}_\mathcal{S}; \mathbb{Z})$};
	\node (D) at (6,1) {$H_0 (\Delta_g; \mathbb{Z})$};
	\node (ad)  at (7.3,1) {};
	\node (B1) at (2, 0) {$\mathbb{Z}^{Q_0}$};
	\node (C1) at (4,0) {$\mathbb{Z}^{n + 1}$};
	\path[->,font=\scriptsize]
	(A) edge node[above] {} (B)
	(B) edge node[above] {$\delta$} (C)
	(C) edge  node[above] {} (D)
	(D) edge node[above] {} (ad)
	(B) edge node[left] {$\cong$} (B1)
	(C) edge node[right] {$\cong$} (C1)
	(B1) edge node[above] {$\phi$} (C1);
	\end{tikzpicture} 
	\end{equation*}
Thus in this case we recover the $\mathbb{V}$-collection from Example~\ref{ex:An4}.
\end{example}

\subsection{Categorical framework for 2d-4d wall-crossing formulas and Hitchin 
integrable systems} \label{sec:hitchin}
We briefly recall the well-known setup for the Hitchin integrable system.
One can interpret the data for such a system  as a surface factorization defined in the following way. Let $Y = C$ be the base curve of the Hitchin system and $S = T^* C$ is its cotangent bundle. Take $K$ to be the canonical bundle on $C$ and $\mathcal{B} = \oplus_{i = 2}^n H^0 (C, K^{\otimes i})$ the set of tuples of differentials and $u = (\phi_2, \ldots, \phi_n) \in \mathcal{B}$. The spectral curve $\Sigma_u$ equals the set of solutions in $T^* C$ to the equation $\sum_{i = 2}^n \phi_i \lambda^{n - i} = 0$. The local solutions to the equation yield $n$-differential forms $\lambda_i$ on $C$ which pullback to a single form $\lambda \in \Omega^1 (C)$.

Suppose $L$ is a line bundle for which $L^{\otimes 2} = K^{\otimes n}$ and define 
\[ X = \left\{ (x_1, x_2, \lambda) \in L \oplus L \oplus K : \sum_{i = 2}^n \phi_i \lambda^{n - i} = x_1^2 + x_2^2 \right\} \] 
When the $\phi_i$ are meromorphic forms there are modifications of the above construction see e.g. \cite[Section~8]{ks08} or \cite[Section~3]{smith15}. Letting $g$ be the projection to $\lambda$, $h$ the bundle projection and $f$ their composition, we obtain a surface factorization of $f$ for generic choices of $u$.

This surface factorization leads to a strictly 4d set-up, and to obtain the full 2d-4d system, one must add a basepoint $z \in C$ which gives us a reference fiber for the Fukaya-Seidel category $\fs{f}{z}$. Using the choice of basepoint gives us the standard $\mathbb{V}_\mathcal{S}$-collection. Integrating the $1$-form $\lambda$ then yields a central charge
\begin{align}
Z_\gamma := \frac{1}{\pi} \int_\gamma \lambda. 
\end{align}
To define the slicing, one utilizes WKB, or spectral networks as defined in \cite{gmn12, gmn13b, gmn13}. Recall that a spectral network is a collection of curves on $C$, and in this case we assume they are WKB curves. Such a WKB curve of phase $\theta$ comes with a pair of lifts to $\Sigma_u$ such that $\arg (\left< h_* (\lambda_i) - h_* (\lambda_j) , \partial_t \right> ) = \theta$ along the curve where $\lambda_i$ and $\lambda_j$ are the two local forms for $\lambda$. A finite spectral network consists of a finite trivalent graph on $C$ whose edges are WKB curves satisfying the conditions:
\begin{enumerate}
	\item at any trivalent vertex, the three adjacent edges are labeled by $ij$, $jk$ and $ki$,
	\item any leaf must be a critical value or $z$
	\item any critical value leaf $q$ must be labeled by $ij$ where, near the critical point $p$ over $q$, the vanishing cycle of the curve is $ij$.
\end{enumerate} 
The networks we will consider will also satisfy a the following additional admissibility conditions:  
\begin{enumerate}
	\item[(4)] for any point $c$ in the interior WKB curve labeled by $ij$, there exists no $k \in \Sigma_u$ lying on the line segment between $i$ and $j$ in $T_c^* C$. 
	\item[(5)] over any trivalent vertex, the triangle $ijk$ in the cotangent fiber does not contain any $l \in \Sigma_u$ lying in its interior.
\end{enumerate}
Note that these conditions are vacuous for a type $A_1$-system.
To any such admissible network, one may define an embedded Lagrangian in $X$ as follows. For each WKB curve $\delta$ in a given network which has an endpoint on a critical value, we take the vanishing thimble $\vt_\delta$ of $f$. Note that by condition (4), we may assume that the vanishing cycle over any point on the interior of such a curve is the matching cycle over a matching path isotopic to the straight line segment from $i$ to $j$. Consider a trivalent vertex $\tilde{z}$ of WKB curves $\delta_1, \delta_2$ and $\delta_3$ of types $ij$, $jk$ and $ik$. Note that, by property (4), the vanishing cycles $\vc_{\delta_1}, \vc_{\delta_2}$ are matching cycles over straight line curves connecting $i$ to $j$ and $j$ to $k$ respectively. Concatenating these paths gives the Lagrangian sum $\vc_{\delta_1} \# \vc_{\delta_2}$. By property (5), after performing an isotopy near the trivalent vertex, this yields a matching cycle over the line segment connecting $i$ to $k$. It is thus clear that over every point on a spectral network there is a well defined matching cycle. Assembling these Lagrangians over the spectral network yields a Lagrangian submanifold of $X$. 
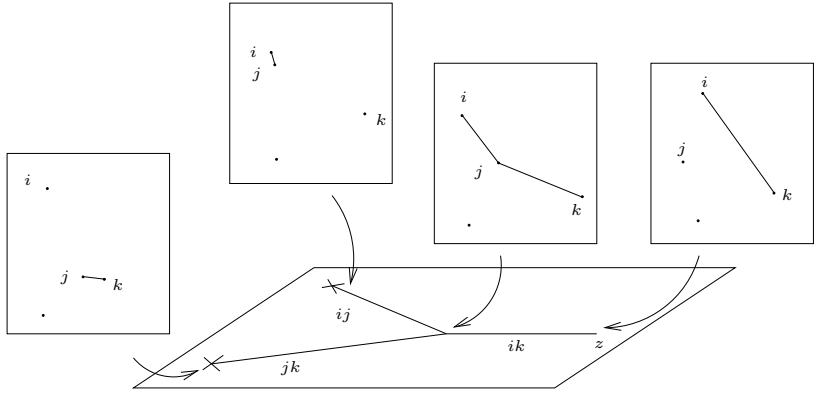
\begin{figure}
	\begin{equation*}
	\begin{tikzpicture}[cross line/.style={preaction={draw=white, -, line width=6pt}}]
	\node[anchor=south west,inner sep=0] (image) at (0,0) {\includegraphics[scale=.8]{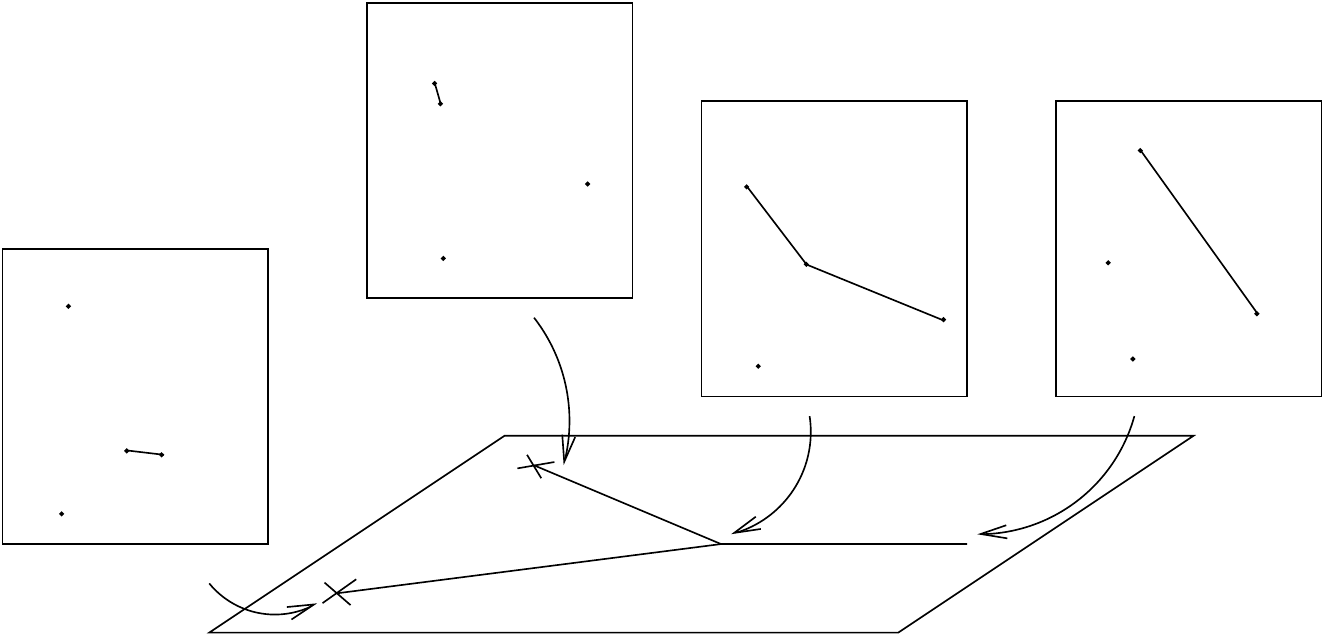}};
	\node[align=center] at (5,3.6) {\tiny $k$};
	\node[align=center] at (3.3,4.5) {\tiny $i$};
	\node[align=center] at (3.35,4.2) {\tiny $j$};
	\node[align=center] at (.3,2.8) {\tiny $i$};
	\node[align=center] at (.8,1.5) {\tiny $j$};
	\node[align=center] at (1.5,1.4) {\tiny $k$};

\node[align=center] at (7.6,2.4) {\tiny $k$};
\node[align=center] at (6.1,3.9) {\tiny $i$};
\node[align=center] at (6.3,2.9) {\tiny $j$};

\node[align=center] at (10.4,2.6) {\tiny $k$};
\node[align=center] at (9.3,4.1) {\tiny $i$};
\node[align=center] at (9,3.2) {\tiny $j$};
	
	\node[align=center] at (3.8,.3) {\tiny $jk$};
	\node[align=center] at (4.5,1) {\tiny $ij$};
	\node[align=center] at (6.8,.6) {\tiny $ik$};
	\node[align=center] at (7.9,.6) {\tiny $z$};
	\end{tikzpicture} 
	\end{equation*}
	\caption{\label{fig:admis} An admissible spectral network.}
\end{figure}

Then we may define the following collection of subcategories of $\fs{z}{f}$. For distinct objects $i, j \in Ob (\mathbb{V}_\mathcal{S} )$ and $\theta \in \mathbb{R}$, let $\mathcal{P}_{ij} (\theta )$ consist of all such Lagrangian branes obtained from spectral networks of phase $\theta$ whose edge incident to $z$ is labeled by $ij$. It is not hard to see that $\mathcal{P}_{ij}$ is a subcategory of $\mathcal{C}_{ij}$. 
\begin{conjecture}
	The collection $\mathcal{P} = \{\mathcal{P}_{ij} (\theta )\}$ form a pre-slicing of the standard $\mathbb{V}_\mathcal{S}$-collection and $(Z , \mathcal{P})$ a $\mathbb{V}_\mathcal{S}$-stability condition.
\end{conjecture}

In particular, the Conjecture means that the Lagrangian submanifolds associated 
with spectral networks are special.

\end{document}